\documentclass[12pt,reqno]{amsart}
\usepackage{amsfonts,amsthm,amsmath,amssymb,amscd}
\usepackage{mathabx} 
\usepackage{esint}
\usepackage{graphics}
\usepackage{indentfirst}
\usepackage{latexsym}
\usepackage[dvips]{epsfig}
\usepackage{mathrsfs}
\usepackage{color}
\usepackage{enumitem}
\usepackage{hyperref}
\usepackage{color, soul}
\date{}
\allowdisplaybreaks \allowdisplaybreaks[2]

\numberwithin{equation}{section}

\usepackage{graphics}
\usepackage{epsfig}
\newcommand{\norm}[1]{\left\lVert #1 \right\rVert}
\newcommand{\abs}[1]{\left\lvert #1 \right\rvert}
\newcommand{\subeqref}[2]{$ \eqref{#1}_{#2} $}
\newcommand{\jump}[1]{\left\ldbrack#1\right\rdbrack} 
\def\p{\partial}
\def\R{\mathbb{R}}
\def\nb{\bar{n}}
\def\mb{\bar{m}} 
\def\ub{\bar{u}}
\def\wb{\bar{w}}
\def\phib{\bar{\phi}}
\def\nt{\tilde{n}}
\def\ut{\tilde{u}}
\def\mt{\tilde{m}}
\def\phit{\tilde{\phi}}
\def\nsinf{n^s_{\infty}}
\def\usinf{u^s_{\infty}}
\def\msinf{m^s_{\infty}}
\def\phisinf{\phi^s_{\infty}}
\def\X{\mathcal{X}}
\def\Y{\mathcal{Y}}

\def\O{\mathcal{O}}
\def\sR{\mathscr{R}}
\def\sG{\mathscr{G}}

\def\sha{\sharp}

\newtheorem{theorem}{Theorem}[section]
\newtheorem{corollary}[theorem]{Corollary}
\newtheorem{lemma}[theorem]{Lemma}
\newtheorem{proposition}[theorem]{Proposition}
\newtheorem{remark}[theorem]{Remark}

\begin{document}

\title[NSP under periodic perturbations]{Asymptotic stability of shock profiles and rarefaction waves to the Navier--Stokes--Poisson system under space-periodic perturbations}
	
	\author{Yeping~Li}
	\address[Yeping Li]{School of Sciences, Nantong University, Nantong 226019, China}
	\email{\href{mailto:ypleemei@aliyun.com}{ypleemei@aliyun.com}}
	
	\author{Yu~Mei}
	\address[Yu Mei]{School of Mathematics and Statistics, Northwestern Polytechnical University, Xi'an 710129, China}
	\email{\href{mailto:yu.mei@nwpu.edu.cn}{yu.mei@nwpu.edu.cn}}
	
	\author{Yuan~Yuan}
	\address[Yuan Yuan]{School of Mathematical Sciences, South China Normal University, Guangzhou 510631, China. }
	\email{\href{mailto:yyuan2102@m.scnu.edu.cn}{yyuan2102@m.scnu.edu.cn}}



\maketitle
\begin{abstract}
This paper concerns with the large-time behaviors of the viscous shock profile and rarefaction wave under initial perturbations which tend to space-periodic functions at infinities for the one-dimensional compressible Navier--Stokes--Poisson equations. It is proved that: (1) for the viscous shock with small strength, if the initial perturbation is suitably small and satisfies a zero-mass type condition, then the solution tends to background viscous shock with a constant shift as time tends to the infinity, and the shift depends on both the mass of the localized perturbation, and the space-periodic perturbation; (2) for the rarefaction wave, if the initial perturbation is suitably small, then the solution tends to background rarefaction wave as time tends to infinity. 

The proof is based on the delicate constructions of the quadratic ansatzes, which capture the infinitely many interactions between the background waves and the periodic perturbations, and the energy method in Eulerian coordinates involving the effect of self-consistent electric field.	Moreover, an abstract lemma is established to distinguish the non-decaying terms and good-decaying terms from the error terms of the equations of the quadratic ansatzes, which will be benefit to constructing the ansatzes and simplifying calculations for other non-localized perturbation problems, especially those with complicatedly coupling physical effects.

\end{abstract}


%
\section{Introduction}
The one-dimensional compressible Naiver--Stokes--Poisson (denoted as NSP in the sequel) equations read
\begin{equation}\label{NSP-Ion}
	\begin{cases}
		n_t+m_x=0,&\quad\\
		m_t+\left(\dfrac{m^2}{n}\right)_x+An_x-n\phi_x=\p_x^2u,&\quad\\
		\p_x^2\phi=n-e^{-\phi},&\quad
	\end{cases}x\in \R, t>0,
\end{equation}
where $n(x,t)>0 $, $ m(x,t)\in\R $ and $u:=m/n$ are the density, the momentum and the velocity of ions, respectively, $ \phi(x,t)$ is the potential of the self-consistent electric field, and the positive constant $A$ is the absolute temperature of ions.  The system \eqref{NSP-Ion} usually models the motion of the viscous ions, in which the electron density $n_e$ is determined by the electrostatic potential $\phi$ in terms of the Boltzmann-Maxwell relation $n_e = e^{ -\phi}$. It could be also extended to the full (non-isentropic) compressible NSP system by taking into account the effect of the temperature, or two-fluid compressible NSP system to describe the motion of the ions and the electrons. It can also governs the motion of viscous ions or charged particles in semiconductor devices, and the motion of self-gravitational viscous gaseous stars in plasmas physics. We refer to \cite{chandrasekhar1957,degond2000,jungel2001,sitenko1995} for more physical background of compressible NSP system.

In this paper, we consider the long time behaviors of the solutions to the Cauchy problem of NSP \eqref{NSP-Ion} under space-periodic perturbation at far fields, which is supplemented with the following initial data 
\begin{equation}\label{initial}
	[n,m](x,0)=[n_0(x),m_0(x)],~~x\in\R,
\end{equation} 
with far fields 
\begin{equation}
	\label{cond-farfield}
	\lim\limits_{x\rightarrow\pm\infty}[n_0(x),m_0(x)]=[\nb^\pm,\mb^\pm]+[\rho_0^\pm(x),w_0^\pm(x)].
\end{equation}
Here $\nb^\pm>0, \mb^\pm, \ub^\pm=\mb^\pm/\nb^\pm$ are different constants, $[\rho_0^\pm,w_0^\pm]$ are periodic functions with periods $\pi^\pm$ and zero average. The localized perturbations are also allowed in the condition \eqref{cond-farfield}.
The far field of potential $\phib^\pm$  is assumed to satisfy the following quasineutral condition
\begin{equation}
	\label{cond-farfield-phi}
\phib^\pm=-\ln \nb^\pm.
\end{equation} 

The constant states $(\nb^\pm,\ub^\pm)$ determine the background waves, and the large time behaviors of the solutions to the Cauchy problem of NSP \eqref{NSP-Ion}-\eqref{cond-farfield-phi}.
When $(\nb^\pm,\ub^\pm)$ are different, there could exist a shock or rarefaction wave to the Riemann problem of the vanishing viscosity and quasi-neutral limit of the above NSP system, i.e., the corresponding Euler equation:
\begin{equation}\label{EP-Ion}
	\begin{cases}
		\p_t n+\p_x(nu)=0,&\\
		n(\p_tu+u\p_xu)+A\p_xn-n\p_x\phi=0,&\\
		\phi=-\ln n,&
	\end{cases}
\end{equation}
with initial data
\begin{equation}\label{Quasi-Ini}
	[n,u](x,0)=[n_0,u_0](x)=\begin{cases}
		[\nb^-,\ub^-],&x>0,\\
		[\nb^+,\ub^+],&x<0.
	\end{cases}
\end{equation}
The quasineutral Euler system \eqref{EP-Ion}-\eqref{Quasi-Ini} has two genuinely nonlinear characteristics 
\begin{equation}
	\lambda_1=u-\sqrt{A+1},\quad\lambda_2=u+\sqrt{A+1}.
\end{equation}
If $(\nb^\pm,\ub^\pm)$ satisfy the Lax shock condition 
\begin{equation}
	\label{entropy}
	\sqrt{(A+1)\nb^+}<s<\sqrt{(A+1)\nb^-},
\end{equation}
and the following Rankine-Hugoniot condition 
\begin{equation}\label{RH}
	\begin{cases}
		&-s\jump{\nb}+ \jump{\nb \ub }=0,\\
		&-s\jump{\nb\ub} + \jump{\nb \ub^2+ (A+1) \nb } =0,
	\end{cases}
\end{equation}
where $\jump{f}:=f^+ - f^-$ for any function $f$, and $s$ is the shock speed, then the Riemann problem \eqref{EP-Ion} and \eqref{Quasi-Ini} admits a 2-shock.
In addition, when the constant states are close enough (i.e. $\abs{\nb^+-\nb^-}$ is small), \eqref{NSP-Ion} has a unique (up to a shift) viscous shock profile $[n^s,m^s, \phi^s](\xi)=[n^s,m^s,\phi^s](x-st)$, which tends to $[\nb^\pm,\mb^\pm,\phib^\pm]$ as $\xi\rightarrow\pm \infty$ (\cite{duan2020a}); see Lemma \ref{lem-shock-prop} below. 
While if the states at far fields satisfy 
\begin{align*}
	&R_1(\nb^-,\ub^-)\equiv\{[n,u]\in\R_+\times\R:u+\sqrt{A+1}\ln n=\ub^-+\sqrt{A+1}\ln \nb^-,n<\nb^-,u>\ub^-
	\}
\end{align*}
or
\begin{align*}
R_2(\nb^-,\ub^-)\equiv\{[n,u]\in\R_+\times\R:u-\sqrt{A+1}\ln n=\ub^--\sqrt{A+1}\ln \nb^-,n>\nb^-,u<\nb^-
	\},
\end{align*}
then the Riemann problem \eqref{EP-Ion} and \eqref{Quasi-Ini} admits a rarefaction wave solution, without loss of generality, taking the second family for example, 
\begin{equation}\label{rare-sol}
	\begin{aligned}
		u^R(x/t)&=w^R(x/t)-\sqrt{A+1},\\
		n^R(x/t)&=\nb^- \exp\left(\dfrac{u^R(x/t)-\ub^-}{\sqrt{A+1}}\right),\\
		\phi^R(x/t)&=-\ln n^R(x/t),
	\end{aligned}
\end{equation}
where $w^R(x/t)$ is the global entropy weak solution to the Riemann problem of the Burgers equation
\begin{equation}
	\begin{cases}
		\p_tw+ww_x=0,\\
		w(x,0)=\begin{cases}
			\wb^-=\ub^-+\sqrt{A+1}\quad \text{for}~ x<0,\\
			\wb^+=\ub^++\sqrt{A+1}\quad \text{for}~x>0.
		\end{cases}
	\end{cases}
\end{equation}

There are substantial literatures on the stability and large-time behaviors around nonlinear wave patterns for gas dynamical equations and other related equations.
For the compressible Navier-Stokes system and the hyperbolic conservation laws, Matsumura and Nishihara \cite{Matsu.N1985}, Kawashima and Matsumura \cite{Kawas.M1985} and Goodman \cite{Goodm.1986} proved the stability of a single viscous shock by energy methods, provided that the shock-strength is small and the initial data carries no excessive mass. Szepessy and Xin \cite{Szepe.X1993} and Liu and Zeng \cite{Liu.Z2015} further removed the zero-mass condition by introducing diffusion waves propagating along other families of characteristics and establishing the point-wise estimates. 
If the shock-strength is arbitrarily large, Zumbrun, Howard, and Mascia \cite{Zumbr.H1998,Masci.Z2004} showed the nonlinear stability under a spectral stability for system of conservation laws and hyperbolic-parabolic systems, and then extended the theory for the Navier-Stokes equations in \cite{Humph.L.Z2017,Humph.L.Z2009,Barke.Z2016}. 
Recently, \cite{He.H2020} successfully used the elementary energy method to obtain the nonlinear stability of the large-amplitude viscous shock for the isentropic Navier-Stokes equations, with the aid of the effective velocity in \cite{Vasse.Y2016}.  
Kang and Vasseur recently developed the $L^2$-contraction properties for conservation laws \cite{Kang.V2017,Kang.V2020a,Kang.V2021}, and together with Wang \cite{Kang.V.W2023}, they proved the stability of the composite wave of a viscous shock and a rarefaction wave for Navier--Stokes equations, which is a long-existing open problem proposed by Matsumura and Nishihara in 1992 \cite{Matsu.N1992}.
We also refer to \cite{Matsu.N1986,Goodm.1989,Xin.1990,Liu.X1997,Nishi.N1999,Huang.X.Y2008,Huang.L.M2010,Li.W.W2018,Li.W.W2020,Wang.W2022,Yuan.2023} and the references therein for the stability results of other Riemann solutions such as rarefaction waves, contact discontinuities and some composite waves and planar waves.

For the compressible NSP system, Duan and Liu \cite{Duan.L2015a} proved the stability of rarefaction waves of the one-dimensional isentropic NSP system for one-fluid or two fulids under small perturbations. Together with Zhang, they \cite{duan2020a} further obtained the existence and stability of viscous shock waves for the same system for one-fluid describing the ions. The stability of rarefaction waves with large initial perturbations was established by Zhang, Zhao and Zhao \cite{zhang2022b} recently.
There are also asymptotic stability of the superposition of a stationary solution and a rarefaction wave for the outflow problem in one-dimensional case ( Li and Zhu \cite{li2017}), and the planar rarefaction wave in three-dimensional case (Li, Luo and Wu \cite{Li.L.W2022}), the contact discontinuity (Liu, Yin and Zhu \cite{liu2016b}) for the full compressible NSP system of the ions. We refer the interested readers to \cite{cui2016,duan2016,hong2018,li2010,wang2010a,wu2017} for more results in the NSP system. 

The study of the periodic perturbations to the hyperbolic conservation laws is also important and interesting. There could exist \emph{new phenomenons} with the non-localized perturbations, and during the study of infinitely many wave interactions, new techniques are developed and more observations of the structures of conservation laws are revealed. 
Xin, Yuan, and Yuan\cite{Xin.Y.Y2019,Yuan.Y2020,Xin.Y.Y2021} first established the stability of shocks and rarefaction waves with periodic perturbations at far fields for the 1-d scalar conservation laws in both inviscid and viscous cases. 
In the inviscid case \cite{Xin.Y.Y2019,Yuan.Y2020}, it is proved that the resultant solutions of the shock or rarefaction waves under periodic perturbations are \emph{exactly} two periodic parts jointed by the background waves. 
While in the viscous case \cite{Xin.Y.Y2021}, the authors proposed a new \emph{quadratic} ansatz, different from other linear ansatzes in localized perturbation problems, to capture the infinitely many interactions of the background wave with the periodic oscillations. 
Moreover, it is firstly found that the long-time limits of the inviscid shock and viscous shock both have shifts induced by the periodic perturbations, \emph{other than} the mass of the localized perturbations. It is a new phenomenon that even though the final shift would not be changed by a perturbation of zero mass in a single period, but it will be \emph{cumulatively} affected by the same perturbations of zero mass with infinitely many periods.
Later, Huang and Yuan \cite{Huang.Y2021} extended this theory to the nonlinear stability of a viscous shock for the 1-d isentropic compressible Navier--Stokes equations under a zero-mass type condition involving periodic perturbations. The nonlinear stability of planar viscous shocks or rarefaction waves for multi-dimensional Navier--Stokes equations, and some other results on composite waves can be found in \cite{Huang.X.Y2022, Yuan.2022a,Yuan.Y2022} and the references therein.

\vline

In this paper, we investigate the coupling effects of the self-consistent electric field and the infinitely many wave interactions induced by non-localized perturbations.
It is proved that: (1) for the viscous shock with small strength, if the initial perturbation is suitably small and satisfies a zero-mass type condition, then the solution to equations \eqref{NSP-Ion}-\eqref{cond-farfield-phi} tends to background viscous shock with a constant shift as time tends to the infinity, which the shift depends on both the mass of the localized perturbations, and the space-periodic perturbations; (2) for the rarefaction wave, if the initial perturbation is suitably small, then the solution to equations \eqref{NSP-Ion}-\eqref{cond-farfield-phi} tends to background rarefaction wave as time tends to infinity. The rigorous theorems are presented in Theorems \ref{thm-shock} and \ref{thm-rare} blew after reformations.

Now let us briefly illustrate the main difficulties on our problem and the key ideas to solve them. The first difficulty arose from the periodic perturbations: the solution to \eqref{NSP-Ion} and \eqref{cond-farfield} is expect to approach the oscillating periodic solutions at infinities (see \cite{Xin.Y.Y2021} for the scalar conservation laws), and thus the difference between the solution and the background shock profile or rarefaction wave is \emph{not integrable} on $\R$, so that the energy method in \cite{Duan.L2015a} and \cite{duan2020a} fails for these non-localized perturbation problems. To overcome this difficulty, the key point is to construct suitable ansatzes which carries the same oscillations as the periodic solutions at the far fields, at the same time such that the error terms of the equations of the ansatzes are well controlled. Motivated by \cite{Xin.Y.Y2021} and \cite{Huang.Y2021}, the ansatzes are constructed as the \emph{quadratic} combinations of the background Riemann solution and the periodic solutions, which capture the infinitely many interactions between the background waves and the periodic perturbations, and are different from the linear ansatzes in the problems of localized perturbations. 

Another difficulty arose from the self-consistent electric field: the construction of the ansatz of the potential is subtle. Although that the potential of the electric field is completely determined by the density with an elliptic equation,  it should not be set as the logarithm of the ansatz of the density as \cite{Duan.L2015a} for the rarefaction wave. 
It should not be a quadratic combination of the weight functions and the periodic solutions as \cite{Xin.Y.Y2021} and \cite{Huang.Y2021}, either. It should be the summation of the background potential and the quadratic parts, which consists of the weight functions induced by the background density and the periodic \emph{oscillations}. Indeed, for the scalar conservation laws \cite{Xin.Y.Y2021} or the isentropic Navier--Stokes \cite{Huang.Y2021}, these last two ways are equivalent. See Remark \eqref{rem-ansatz} below for more details. The setting of the ansatz can be interpreted by a simple fact that it is the periodic oscillations, not directly the periodic solutions, that interact with the background waves at far fields. 

Compared with the classical localized perturbation problems of the Navier--Stokes system, both the nonlinear ansatzes and the coupling effects of the potential arise more complicated error terms of the equations of the ansatzes.  
The study of the ansatz of the coupling potential and the complicated error terms motivates us to propose an abstract lemma which distinguishes the non-decaying terms and good-decaying terms from the error terms of the equations of the quadratic ansatzes. 
It is revealed by the lemma that the quadratic ansatzes can be well defined such that not only the ansatzes carry the same oscillations as the periodic solutions at one far field and disappear at the other far field, but also the time-decaying  and spatially non-localized functions in the error terms are multiple of the localized weights.
This lemma will be of benefits for other periodic perturbation problems, especially those with other complicatedly coupling physical effects, to determine the ansatzes, and simplify the calculations.

The rest of the paper is organized as follows. In Section 2, we introduce the notations and prepare the properties of the background waves and the quadratic ansatzes. The constructions of the ansatzes and the main theorems are presented in Section 3. The proof about the stabilities of the shock profile and rarefaction wave are placed in Section 4 and 5, respectively. The proofs of Lemma \ref{Lem-periodic} about the exponential decay of the periodic perturbations and Lemma \ref{lem-error} about the quadratic ansatzes are supplied in the appendix.

%
\section{Preliminaries}
In this section, we first introduce the notations and recall some basic properties of periodic solutions, viscous shock profiles, and rarefaction waves to \eqref{NSP-Ion}, and finally present the abstract lemma about the quadratic ansatzes.

Throughout this paper, $L^p(U)$ and $H^{k}(U)$ ($1\leq p\leq \infty$) denote the classical Lebesgue space, Sobolev space, respectively. Without indications on the sets, the spaces are regarded as the ones of the functions on $\R$ and $\norm{\cdot}$ denotes for $L^2$-norm on $\R$.

\subsection{Decay of periodic solutions} 
Let $[n^\pm,m^\pm,\phi^\pm](x,t)$ be the unique periodic solution to \eqref{NSP-Ion} with the periodic initial data $$[n^\pm,m^\pm](x,0)=[\nb^\pm,\mb^\pm]+[\rho_0^\pm(x),w_0^\pm(x)],$$ 
respectively, where $[\rho_0^\pm,w_0^\pm]$ are periodic functions on $\pi^\pm$ satisfying 
\begin{equation}
	\label{average0}
	\int_0^{\pi^\pm}[\rho_0^\pm,w_0^\pm] \, dx=0.
\end{equation} 
We denote the periodic perturbations $[\rho^\pm,w^\pm,v^\pm,\varphi^\pm]$ as follows for later use:
\begin{equation}\label{def-pert}
	[\rho^\pm,w^\pm,v^\pm,\varphi^\pm](x,t):=[n^\pm,m^\pm,u^\pm,\phi^\pm](x,t)-[\nb^\pm,\mb^\pm,\ub^\pm,\phib^\pm].
\end{equation} 
The periodic solutions to \eqref{NSP-Ion} have exponential decays: 
\begin{lemma}\label{Lem-periodic}
	Consider the Cauchy problem \eqref{NSP-Ion} with the periodic initial data
\begin{equation}
	\label{initial-periodic}
	[n,m](x,0)=[\nb,\mb]+[\rho_0,w_0](x),
\end{equation}  
	where $\nb, \mb,\ub:=\mb/\nb, \phib:=-\ln \nb$ are constants, and $ [\rho_0,w_0](x)\in H^k\big((0,\pi)\big) $ with $ k\geq 1 $ is periodic with period $ \pi>0 $, and satisfies 
	\begin{equation}
		\int_{0}^{\pi} [\rho_0, w_0]dx =0.
	\end{equation}
	Then there exists $ \nu_0>0 $ such that if $\nu:=\norm{[\rho_0,w_0] }_{H^k((0,\pi))} \leq \nu_0,$ there exits a unique periodic solution $ [n,m,u,\phi](x,t) \in C\big(0,+\infty;H^k((0,\pi)) \big) $, and it satisfies that for $t\geq0$,
	\begin{equation}\label{decay-per}
		\begin{aligned}
			&\quad \int_{0}^{\pi} [n-\nb, m-\mb](x,t) dx =0,\\
			&\norm{[n,m,u,\phi]-[\nb,\mb,\ub,\phib]}_{H^k((0,\pi))}(t) \leq C \nu e^{-2\alpha t}, 
		\end{aligned}
	\end{equation}
	where the constants $ C>0 $ and $ \alpha >0 $ are independent of $ \nu $ and $ t. $
\end{lemma}

The proof of Lemma \ref{Lem-periodic} is based on standard energy method and Poincar\'{e} inequality, which is left in the appendix.

\subsection{Properties of shock profiles and rarefaction waves}
Now we introduce the basic properties of shock profiles and rarefaction waves. For the viscous shock waves of \eqref{NSP-Ion} with small strengths, we have the following lemma.

\begin{lemma}[\cite{duan2020a}]\label{lem-shock-prop}
	Assume that \eqref{RH} and \eqref{entropy} hold. Then there exist $\delta_0>0$ and $C>0$ such that if \begin{equation}
		\label{small-strength}
		\delta:=\abs{\nb^+ - \nb^-} \leq \delta_0.
	\end{equation}
	 holds, then \eqref{NSP-Ion} admits a unique traveling wave solution (up to a shift) $[n^s,m^s, \phi^s](\xi)=[n^s,m^s,\phi^s] (x-st)$ satisfying 
    \begin{equation}
    	\label{eqn-shock}
    	C^{-1} \p_\xi \phi^s \leq \p_\xi n^s = \frac{(n^s)^2 \p_\xi u^s}{\nb^- \abs{\ub^- - s}} \leq C \p_\xi \phi^s<0, 
    \end{equation}
    for any $\xi\in \R$.
	Moreover, by a suitable choice of the shift, the shock profile $[n^s,m^s,\phi^s]$ satisfies
	\begin{equation}\label{eqn-shock-deriv}
		\abs{ \frac{d^k}{d\xi^k} [n^s-\nb^\pm, u^s-\ub^\pm, \phi^s-\phib^\pm ](\xi)} 
		\leq C_k \abs{\nb^+ - \nb^-}^{k+1} e^{-\theta \abs{(\nb^+ - \nb^-)\xi} }
	\end{equation}
	for $\xi \lessgtr 0$ and $k=0,1,\cdots$, where $C_k$ and $\theta$ are generic positive constant.
\end{lemma}
To study the stability of rarefaction waves, we need to construct a smooth approximation as in \cite{Matsu.N1992,Duan.L2015a,Li.L.W2022} since the rarefaction wave solution $[n^R,u^R,\phi^R]$ in \eqref{rare-sol} is only Lipschitz continuous, which is difficult to perform higher order energy estimates. Precisely, we define approximate rarefaction waves $[n^r,u^r,\phi^r]$ by replacing $w^R(x/t)$ in \eqref{rare-sol} with $w^r(x,t+1)$, where $w^r(x,t)$ is an approximate smooth function to the following Burgers equation 
\begin{equation}\label{wr}
	\begin{cases}
		\p_tw+ww_x=0,\\
		w(x,0)=\dfrac{\wb^++\wb^-}{2}+\dfrac{\wb^+-\wb^-}{2}\tanh(\epsilon x).
	\end{cases}
\end{equation}
and $\epsilon>0$ is a small parameter to be determined. It is easy to check that $[n^r,u^r]$ satisfies \eqref{EP-Ion} and $[n^r,u^r](x,0)\rightarrow [\nb^\pm,\ub^\pm]$ as $x\rightarrow\pm \infty$. We have the following lemma for the properties of approximate rarefaction waves $[n^r,u^r,\phi^r]$. 
\begin{lemma}[\cite{Li.L.W2022}]\label{lem-rare-prop}
	The approximate smooth 2-rarefaction wave $(n^r, u^r, \phi^r)$ given by \eqref{rare-sol} and \eqref{wr}  satisfies the following properties:
	\begin{enumerate}[label = \rm{(\roman*)},ref =\rm{(\roman*)}]
		\item $\p_x n^r>0, \p_x u^r>0$ and $\nb^-<n^r(x,t)<\nb^+$, $\ub^-<u^r(x,t)<\ub^+$ for $x\in \R$ and $t\geq0$.
		\item For any $1\leq p\leq +\infty$, there exists a constant $C_p$ such that for all $t>0$, 
		\begin{align*}
			&\norm{ \p_x [n^r,u^r,\phi^r]}_{L^p} \leq C_p \min\{ \delta_r \epsilon^{1-1/p}, \delta_r^{1/p} t^{-1+1/p} \},\\
			&\norm{ \p_x^2 [n^r,u^r,\phi^r]}_{L^p} \leq C_p \min\{ \delta_r \epsilon^{2-1/p}, \epsilon^{1-1/p} t^{-1} \},\\
			&\norm{ \p_x^3 [n^r,u^r,\phi^r]}_{L^p} \leq C_p \min\{ \delta_r \epsilon^{3-1/p}, \epsilon^{2-1/p} t^{-1} \},
		\end{align*}
		where $\delta_r=\abs{\nb^+ -\nb^-}+\abs{\ub^+ -\ub^-}.$
		
		
		\item $\lim\limits_{t\rightarrow +\infty} \sup\limits_{x\in \R} \abs{[n^r,u^r,\phi^r](x,t)-[n^R,u^R,\phi^R](x/t)}=0.$
	\end{enumerate}
\end{lemma}

\subsection{Properties of error terms}
Before introducing the detailed formulas of the ansatzes, we prepare the following notations and the lemma to distinguish the non-decaying terms and good-decaying terms from the error terms of the equations of the quadratic ansatzes. It could be benefit to other non-localized perturbation problems.

Throughout the whole paper, we denote $C$ as the generic constant, $\sR_{\rm Ti}$ as the generic terms satisfying
\begin{equation}
	\label{decay-Ti}
	\tag{RT}
	\norm{\sR_{\rm Ti}}_{L^{\infty}} \leq C \nu e^{-\alpha t},
\end{equation} 
which decay exponentially to $0$ in time but may be not integrable in space (the periodic perturbations are  $\sR_{\rm Ti}$, i.e. \eqref{decay-per});
denote $\sR_{\rm Sp}$ as the generic terms satisfying
\begin{equation}
	\label{decay-Sp}
	\tag{RS}
	\norm{\sR_{\rm Sp}}_{L^p} \leq C (1+t)^{1/p}
\end{equation} 
for any $1\leq p<\infty$, which are integrable in space and may not decay in time;
denote $\sR$ as the generic terms in the quadratic form $\sR=\sR_{\rm Ti} \sR_{\rm Sp}$, and thus $\sR$ are good-decaying terms:
\begin{equation}
	\label{decay-TiSp}
	\norm{\sR}_{L^p}\leq C \nu e^{-\alpha t}.
\end{equation}
We define $\sG$ as a family of regular and bounded weight functions on $x\in \mathbb{R}$ such that 
\begin{equation} \label{def-sG}
	\begin{aligned}
		&\text{if}~ g\in \sG,~ \text{then} \lim\limits_{x\rightarrow+\infty}g=1, \lim\limits_{x\rightarrow-\infty}g=0,~\text{and}~(1-g)g,~ \p_x g   ~\text{satisfy \eqref{decay-Sp}};\\
		&\text{if}~ g_1, g_2\in \sG,~ \text{then both}~ (1-g_1)g_2~\text{and}~ g_1-g_2 ~\text{satisfy \eqref{decay-Sp}}.
	\end{aligned}
\end{equation}
In particular, $\sigma$ defined in \eqref{weight-shock} below and its shifts  $\sigma_{\X},\sigma_{\Y}$ with $\X,\Y$ given in Lemma \ref{Lem-shift} all belong to $\sG$.

The following lemma will show that after the operations of compositions, derivations, linear combinations and multiplications on the quadratic ansatzes, the non-decaying terms are the summation of the background profile and the quadratic parts which carries the same oscillations at the far fields. And the remaining terms are in the form as $\sR_{\rm Ti}\sR_{\rm Sp} $, and thus could be well controlled. 

Just in the following lemma and corollary, we temporarily redefine the notations and give the assumptions:
\begin{center}
	\begin{tabular}{|c|l|}
		$\ub^\pm $ ($\nb^\pm$) & constants \\
		$v^\pm$ ($\rho^\pm$) &  functions on $(x,t)$, and $ v^\pm, \p_x v^\pm$ ($ \rho^\pm, \p_x \rho^\pm$) belong to $\sR_{\rm Ti}$, i.e. satisfy \eqref{decay-Ti} \\
		$u^\pm$  ($n^\pm$) &  $u^\pm:=\ub^\pm+v^\pm$ ($n^\pm:=\nb^\pm+\rho^\pm$) \\
		$u^*$  ($n^*$) &   function on $(x,t)$ with values between $\ub^\pm$  ($\nb^\pm$), $ u^* \rightarrow \ub^\pm$ ($ n^* \rightarrow \nb^\pm$)  as $x\rightarrow\pm\infty$ \\
		$u^\sha$  ($n^\sha$) &   function on $(x,t)$ with values between $u^\pm$  ($n^\pm$), $ u^\sha \rightarrow u^\pm$ ($ n^\sha \rightarrow n^\pm$)  as $x\rightarrow\pm\infty$ \\
		$g_0$  ($\tilde{g}_0$) &  $g_0:=(u^*-\ub^-)/(\ub^+-\ub^-)$ ($\tilde{g}_0:=(n^*-\nb^-)/(\nb^+-\nb^-)$) $\in \sG$ 
	\end{tabular}
\end{center}
where the last assumptions that $g_0, \tilde{g}_0 \in \sG$ actually require that $u^*$ and $n^*$ are well decaying to $\ub^\pm$ and $\nb^\pm$ as $x\rightarrow \pm\infty$, respectively.

\begin{lemma}\label{lem-error}
	Given the definitions and assumptions above, if there are weight functions $g, \tilde{g}, g_1$ belong to $\sG$ and
	\begin{equation}
		\label{generic-usha}
		u^\sha- u^*-v^-(1-g)-v^+ g=\sR, \quad n^\sha- n^*-\rho^-(1-\tilde{g})-\rho^+ \tilde{g}=\sR,
	\end{equation}
	then the following conclusions hold:
	\begin{enumerate}[label = \rm{(\roman*)},ref =\rm{(\roman*)}]
		\item 
		For any $C^2$ function $f$, there hold 
		\begin{equation}
			\label{lem-error-1}
			f(u^\sha)-f(u^*)-(f(u^-)-f(\ub^-))(1-g_1)-(f(u^+)-f(\ub^+))g_1=\sR,
		\end{equation}
		\begin{equation}
			\label{lem-error-2}
			f(u^\sha)-f(u^-)(1-g_1)-f(u^+)g_1=\sR_{\rm Sp},
		\end{equation}
		\eqref{lem-error-2} equals to $\sR$ if  $g_1=\dfrac{f(u^*)-f(\ub^-)}{f(\ub^+)-f(\ub^-)}$.
		
		\item Suppose that the derivative of the error term of \eqref{generic-usha} is still in the form $\sR_{\rm Ti} \sR_{\rm Sp}$, i.e. satisfies \eqref{decay-TiSp}. Then
		\begin{equation}
			\label{lem-error-3}
			\p_xu^\sha-\p_xu^*-\p_x u^-(1-g_1)-\p_x u^+g_1=\sR.
		\end{equation}
		
		\item  For any $\alpha, \beta \in \R$,
		\begin{equation}
			\label{lem-error-3.5}
			\alpha n^\sha + \beta u^\sha-\alpha n^*-\beta u^*-(\alpha\rho^- + \beta v^- )(1-g_1)-(\alpha\rho^+ + \beta v^+ )g_1=\sR,
		\end{equation}
		\begin{equation}
			\label{lem-error-4}
			n^\sha u^\sha-n^*u^*-(n^- u^- -\nb^-\ub^-)(1-g_1)-(n^+ u^+ -\nb^+\ub^+)g_1=\sR.
		\end{equation}
	\end{enumerate}
\end{lemma}
The proof of Lemma \ref{lem-error} is to verify that the remaining terms are non-localized functions multiple of integrable weights, like $\p_x g$ or $(1-g)g$ for $g\in \sG$. The proof is tedious but not difficult, and thus it is left in the appendix.

It should be noticed that the weight function $g_1$ in \eqref{lem-error-1}-\eqref{lem-error-4} is only required to satisfy the property \eqref{def-sG}, and it may not be related with $g_0, \tilde{g}_0, g,$ and $\tilde{g}$, i.e. related with $u^*, u^\sha, n^*, $ and $n^\sha$. Moreover, all the terms in the equation \eqref{NSP-Ion} could be several combinations of the above four operations. Here is an example to illustrate it.
\begin{corollary}
	Under the assumptions in Lemma \ref{lem-error}, it holds that 
	\begin{equation}
		\label{lem-error-test}
		\p_x\left(\frac{(u^\sha)^2}{n^\sha}\right)-\p_x\left(\frac{(u^*)^2}{n^*}\right)-\p_x\left(\frac{(u^-)^2}{n^-}\right)(1-g_1)-\p_x\left(\frac{(u^+)^2}{n^+}\right)g_1=\sR.
	\end{equation}
\end{corollary}
\begin{proof}
	\eqref{lem-error-test} could be obtained by applying the operation of compositions,  \eqref{lem-error-1}, to $n$ with $f_1(n)=n^{-1}$ for $n>0$, and to $u$ with $f_2(u)=u^{2}$, and then applying the operation of multiplication, \eqref{lem-error-4}, to $n^{-1}$ and $u^2$, and finally applying the operation of derivation, \eqref{lem-error-2}, to $n^{-1} u^2$.
\end{proof}

\section{Ansatzes and main results}
In this section, we will construct the ansatzes, reformulate the problems and state the main results.
\subsection{Ansatz and main result for shock profiles}
\subsubsection{Ansatz and the equations}
Assume that \eqref{RH}, \eqref{entropy} and \eqref{small-strength} hold such that the shock profile exists. 
To study the stability of shock profiles under periodic perturbations at far fields, we impose the initial data for \eqref{NSP-Ion} with the condition more than \eqref{cond-farfield}:
\begin{equation}
	\label{init-shock}
		[n_0-n^s(\cdot,0) -\rho_0^\pm, m_0-m^s(\cdot,0)-w_0^\pm] \in L^1(\R^\pm).
\end{equation} 
where $[\rho_0^\pm,w_0^\pm]$ are periodic functions with zero average \eqref{average0}.

Note that we can decompose $n^s$ into 
 \begin{equation}
 	n^s=\nb^-(1-\sigma)+\nb^+\sigma,\quad 
 \end{equation}
 with
 \begin{equation}
 	\label{weight-shock}
 	\sigma=\dfrac{n^s-\nb^-}{\nb^+-\nb^-},
 \end{equation} 
 Moreover, by using $s=\dfrac{m^s-\mb^-}{n^s-\nb^-}=\dfrac{\mb^+-\mb^-}{\nb^+-\nb^-}$, which follows from the Rankine-Hugoniot  condition of viscous shock and \eqref{RH},we can also decompose $m^s$ into 
 \begin{equation}
 	m^s=\mb^-(1-\sigma)+\mb^+\sigma.
 \end{equation}
Throughout the theorem and the proof of shock profiles, we denote $\X=\X(t),\Y=\Y(t)$ with the property that $\abs{\X}, \abs{\Y}\leq C$ for all $t\geq 0$ as the shift functions to be determined; see \eqref{ode-shift}, \eqref{def-shift0-1} and \eqref{def-shift0-2} below.
We use the convection that when $\X$ or $\Y$ appears in the subscript of a function, it means that the function is shifted by $\X$ or $\Y$ and the background shock, for instance,
$$\sigma_{\X}:=\sigma(x-st-\X);$$
while when $x$ or $t$ is in the subscript, it means that the function is differentiated by $x$ or $t$. 

Inspired by \cite{Xin.Y.Y2021,Huang.Y2021} and Lemma \ref{lem-error}, we construct the ansatzes as
\begin{equation}
	\label{def-ansatz-shock}
	\begin{aligned}
		n^\sha&:=n^-(1-\sigma_{\X})+ n^+\sigma_{\X}
		=n^s_{\X}+\rho^-(1-\sigma_{\X})+\rho^+\sigma_{\X},\\
		m^\sha&:=m^-(1-\sigma_{\Y}) + m^+\sigma_{\Y}
		=m^s_{\Y}+w^-(1-\sigma_{\Y})+w^+\sigma_{\Y},\\ 
		u^\sha&:=m^\sha / n^\sha,\\
		\phi^\sha&:=\phi^s_{\X}+\varphi^-(1-\sigma_{\X})+\varphi^+\sigma_{\X},
	\end{aligned}
\end{equation}
where $[\rho^\pm,w^\pm,v^\pm,\varphi^\pm]$ are defined in \eqref{def-pert}. It should be noticed that $\phi^\sha\neq \phi^-(1-\sigma_{\X})+\phi^+\sigma_{\X}.$ We refer to Remark \ref{rem-ansatz} below for the discussions on the setting of $\phi^\sha$.

Define the perturbation 
\begin{equation}
	\label{def-pert-shock}
	[\nt,\mt,\ut,\phit]=[n-n^\sha,m-m^\sha,u-u^\sha,\phi-\phi^\sha],
\end{equation}
and denote $$[\nsinf,\msinf,\usinf,\phisinf]=[n^s_{\X_\infty}, m^s_{\X_\infty}, u^s_{\X_\infty}, \phi^s_{\X_\infty}]$$ for simplicity.
Then with the aid of Lemma \ref{lem-error}, we obtain that $[\nt,\mt,\phit]$ satisfies
\begin{equation}
	\label{eqn-pert-shock}
	\begin{cases}
		\p_t\nt+\p_x\mt=h_1, \\
		\p_t\mt+\p_x \left(\dfrac{m^2}{n}-\dfrac{(m^\sha)^2}{n^\sha} + (A+1)\tilde{n} \right)
		-\p_x^2\left(\dfrac{m}{n}-\dfrac{m^\sha}{n^\sha}\right)\\
		\qquad
		-\p_x\left(\frac{1}{2}(\p_x \phi)^2-\frac{1}{2}(\p_x \phi^\sha)^2+\p_x^2 \phit \right)=h_2,\\
		\p_x^2\phit-\nt + (e^{-\phi} -e^{-\phi^\sha})=h_3,
	\end{cases}
\end{equation} 
where the error terms $h_1, h_2, h_3$ in \eqref{eqn-pert-shock} are 
\begin{equation}
	\label{eqn-error-shock}
	\begin{aligned}
		h_1=&\left\{ - \jump{n}(s+ \X') + \jump{m} \right\} \p_x\sigma_{\X} + \p_x (m^s_{\Y}-m^s_{\X}) + \p_x \sR, \\
		h_2=& \Big\{ - \jump{m}(s+\Y') + \jump{\frac{m^2}{n}+(A+1)n - \p_x\frac{m}{n}-\frac{1}{2} (\p_x \phi)^2-\p_x^2 \phi} \Big\} \p_x \sigma_{\Y} \\
		&\qquad + \p_x \Big\{ \Big( \frac{(m^s_{\Y})^2}{n^s_{\X}} + (A+1)n^s_{\X}-\p_x \frac{m^s_{\Y}}{n^s_{\X}} -\frac{1}{2} (\phi^s_{\X})^2 -\p_x^2 \phi^s_{\X} \Big) \\
		&\qquad -\Big( \frac{(m^s_{\Y})^2}{n^s_{\Y}} + (A+1)n^s_{\Y}-\p_x \frac{m^s_{\Y}}{n^s_{\Y}} -\frac{1}{2} (\phi^s_{\Y})^2 -\p_x^2 \phi^s_{\Y} \Big)\Big\} + \p_x \sR,\\
		h_3=&\sR.
	\end{aligned}
\end{equation}

The details to derive \eqref{eqn-pert-shock} and \eqref{eqn-error-shock} are left in the appendix Section \ref{subsect-error}. It should be noticed that here $\jump{n}=n^+ - n^-$, $\jump{m}=m^+ - m^-$ are still functions, not just constants. Moreover, thanks to Lemma \ref{lem-error}, good-decaying terms in $h_1$, $h_2$, and $h_3$ are quadratic in the form $\p_x \sR=\p_x (\sR_{\rm Ti} \sR_{\rm Sp})$ or $\sR=\sR_{\rm Ti} \sR_{\rm Sp}$ (we keep $\p_x \sR$ in $h_1$ and $h_2$ rather than $\sR$ since $\p_x \sR$ is of zero mass, which is important for the derivation of $\X, \Y$, \eqref{cond-shift-dt2} and \eqref{ode-shift}, and the estimates of error terms of anti-derivatives, Lemma \ref{h-estimates}). For the stability of shock profiles, $\sR_{\rm Ti}$ are the periodic functions and their derivatives, and $\sR_{\rm Sp}$ are the derivatives of the background shock profiles. In particular, such forms are kept after derivations for several times due to Lemma \ref{Lem-periodic}, \ref{lem-shock-prop} and \ref{lem-rare-prop}. 
The rest terms in $h_1$ and $h_2$ are non-decaying for general $\X, \Y$, but finally decay well with the special settings of $\X, \Y$ below.

\vline

\begin{remark}
	\label{rem-ansatz}
	$\phi^\sha$ could not be defined as $-\ln n^\sha$ as \cite{Duan.L2015a}, nor $\phi^- (1-\sigma_\X)+\phi^+ \sigma_\X$ as \cite{Xin.Y.Y2021} and \cite{Huang.Y2021}.
	Indeed, if $\phi^\sha$ was set as $-\ln n^\sha$, then the difference of the equations of $\phi$ and $\phi^\sha$, i.e. $h_3$ in the above equation, would have $\p_x^2 \phi^\sha$, which has non-decaying periodic perturbations at far fields. 
	If $\phi^\sha$ was set as $\phi^- (1-\sigma_\X) + \phi^+ \sigma_\X$, then $h_3$ would have 
	$$e^{-\phi^\sha}-e^{-\phi^-}(1-\sigma_\X)-e^{-\phi^+}\sigma_\X,$$ which may not decay in time as shown in \eqref{lem-error-2}. 
	For the scalar conservation laws \cite{Xin.Y.Y2021} or the isentropic Navier--Stokes \cite{Huang.Y2021}, all the equations can be written in conservative form, and there is no term like $e^{-\phi}$, and thus the last ansatzes work for their problems.
\end{remark} 

\subsubsection{Settings of the shifts}
We define the anti-derivatives of $[\nt,\mt]$ as 
\begin{equation}
	\label{def-anti}
	[\Phi,\Psi]=\int_{-\infty}^x [\nt,\mt] dy \,, \quad [\Phi_0,\Psi_0]=\int_{-\infty}^x [\nt(y,0),\mt(y,0)] dy.
\end{equation}
To ensure that anti-derivatives in \eqref{def-anti} are well-defined in $L^2$, the shift curves $\X,\Y$ are set to satisfy 
\begin{equation}
	\label{cond-shift}
	\int_\R [n-n^\sha,m-m^\sha] (x,t) dx =0 \quad \text{for all~}t\geq0.
\end{equation}
The above two equations hold when they hold initially and
\begin{equation}\label{cond-shift-dt}
	\frac{d}{dt}\int_\R [n-n^\sha,m-m^\sha] (x,t) dx =0 \quad \text{for all~}t\geq0.
\end{equation}
Thanks to \eqref{eqn-pert-shock}-\eqref{eqn-error-shock}, \eqref{cond-shift-dt} is equivalent to
\begin{equation}
	\label{cond-shift-dt2}
	\int_\R [h_1,h_2] (x,t) dx =0 \quad \text{for all~}t\geq0.
\end{equation}
Noting that $\int_\R \p_x \sR dx =0$, it follows that $\X, \Y$ have to satisfy
\begin{equation}
	\label{ode-shift}
	\left\{
	\begin{aligned}
		&\X' = -s + \frac{ \int_\R \jump{m} \p_x \sigma_{\X} dx }{\int_\R \jump{n} \p_x \sigma_{\X} dx},\\
		&\Y' =-s + \frac{ \int_\R \jump{\frac{m^2}{n}+(A+1)n - \p_x\frac{m}{n}-\frac{1}{2} (\p_x \phi)^2-\p_x^2 \phi} \p_x \sigma_{\Y} dx }{ \int_\R \jump{m} \p_x \sigma_{\Y} dx }.
	\end{aligned}
	\right.
\end{equation}
While the first component of \eqref{cond-shift} at $t=0$ just requires 
\begin{align*}
	0=\int_\R (n_0 -n^s_{\X_0}-\rho_0^-(1-\sigma_{\X_0})-\rho_0^+\sigma_{\X_0}) dx,
\end{align*}
which yields that
	\begin{align}
	&	\mathscr{A}_1(\X_0) + \frac{1}{\jump{\nb}} \left\{ \int_{-\infty}^{0}(n_0-n^s(\cdot,0)-\rho_0^-) dx +\int_{0}^{\infty}(n_0-n^s(\cdot,0)-\rho_0^+) dx\right\}=0,  	
	\label{def-shift0-1} \\
	&\text{where~} \mathscr{A}_1(\X_0)=\X_0 + \frac{1}{\jump{\nb}} \left\{ -\int_{-\infty}^{0} \jump{\rho_0} \sigma_{\X_0}dx +\int_{0}^{\infty}\jump{\rho_0} (1-\sigma_{\X_0}) dx\right\}. \label{def-A1}
	\end{align} 
And similarly, 
\begin{align}
	&\mathscr{A}_2(\Y_0) + \frac{1}{\jump{\mb}} \left\{ \int_{-\infty}^{0}(m_0-m^s(\cdot,0)-w_0^-) dx +\int_{0}^{\infty}(m_0-m^s(\cdot,0)-w_0^+) dx\right\}=0,  \label{def-shift0-2} \\
	&\text{where~}\mathscr{A}_2(\Y_0)=\Y_0 + \frac{1}{\jump{\mb}} \left\{ -\int_{-\infty}^{0} \jump{w_0} \sigma_{\X_0}dx +\int_{0}^{\infty}\jump{w_0} (1-\sigma_{\X_0}) dx\right\}. \label{def-A2}
\end{align}

\begin{lemma}\label{Lem-shift}
		Assume that \eqref{RH} and \eqref{entropy} hold. Then there exists a $\gamma_0>0$ such that if the periodic perturbations $ [\rho_0^\pm, w_0^\pm]$ on are of zero average, i.e. \eqref{average0}, and 
		\begin{equation}
			\label{small-periodic}
			\nu:=\norm{[\rho_0^\pm, w_0^\pm]}_{H^3((0,\pi^\pm))}\leq \gamma_0\delta \quad\text{with~} \delta=\abs{\nb^+ - \nb^-},
		\end{equation}  
		then $\left(\X_0, \Y_0\right)$ is uniquely determined by \eqref{def-shift0-1} and \eqref{def-shift0-2}, and there exists a unique solution $ [\X, \Y](t) \in C^1[0,+\infty) $ to the system \eqref{ode-shift} with the initial data $ [\X, \Y](0)=[\X_0, \Y_0], $ and there holds
		\begin{equation*}
			\abs{\left[\X',\Y'\right](t)} \leq C\nu e^{-2\alpha t}, \qquad t\geq 0,
		\end{equation*}	
	and
			\begin{equation*}
		 \abs{\left[\X-\X_\infty,\Y-\Y_\infty\right](t)} \leq C\nu\delta^{-1} e^{-2\alpha t}\leq C\gamma_0e^{-2\alpha t}, \qquad t\geq 0,
	\end{equation*}	
		where the constant $ \alpha>0 $ is independent of $ \nu $ and $ t. $ 
		Moreover, the corresponding locations $ \X_\infty, \Y_\infty$ can be written in terms of the constants $\X_0,\Y_0$ and the periodic solutions $[n^\pm,m^\pm,\phi^\pm]$ as follows,
		\begin{align}
			\X_\infty =~& \mathscr{A}_1(\X_0) - \frac{1}{\jump{\nb}}  \jump{\frac{1}{\pi}\int_{0}^{\pi} \int_{0}^{x}\rho_0(y)dydx},
			\label{X-inf} \\
			\Y_\infty =~ & \mathscr{A}_2(\Y_0) - \frac{1}{\jump{\mb}} \jump{\frac{1}{\pi}\int_{0}^{\pi} \int_{0}^{x}w_0(y)dydx} 
			+ \frac{1}{\jump{\mb}}\int_{0}^{+\infty}  \Big\ldbrack\frac{1}{\pi}\int_{0}^{\pi} \Big(\frac{m^2}{n} \label{Y-inf} \\
			&\qquad +(A+1)n -\frac{1}{2} (\p_x \phi)^2
			- \frac{\mb^2}{\nb}- (A+1)\nb \Big) dx  \Big\rdbrack  dt.  \notag 
		\end{align}
	\end{lemma}
At the same time, it should be assumed that
\begin{equation}
	\label{coinciding position}
	\X_\infty=\Y_\infty.
\end{equation}
Otherwise, the limit of the ansatz $[n^\sha,m^\sha](x,t)$ as $t\rightarrow \infty$, i.e. $[n^s(x-st-\X_\infty),m^s(x-st-\Y_\infty)]$, 
is not a traveling solution to \eqref{NSP-Ion}, and the error term $h_2$ in \eqref{eqn-error-shock} would not vanish as $t\rightarrow \infty$.
With \eqref{X-inf} and \eqref{Y-inf}, \eqref{coinciding position} is equivalent to the following \emph{zero-mass type condition}:
\begin{equation}
	\label{zero mass}
	\begin{aligned}
		&s\Big\{
		\int_{-\infty}^{0}(n_0-n^s(\cdot,0)-\rho_0^-) dx +\int_{0}^{\infty}(n_0-n^s(\cdot,0)-\rho_0^+) dx
		\\
		&\quad +  \jump{\frac{1}{\pi}\int_{0}^{\pi} \int_{0}^{x}\rho_0(y)dydx} \Big\} 
		\\
		=& \int_{-\infty}^{0}(m_0-m^s(\cdot,0)-w_0^-) dx +\int_{0}^{\infty}(m_0-m^s(\cdot,0)-w_0^+) dx 
		\\
		&\quad +  \jump{\frac{1}{\pi}\int_{0}^{\pi} \int_{0}^{x}w_0(y)dydx} 
		-\int_{0}^{+\infty}  \Big\ldbrack\frac{1}{\pi}\int_{0}^{\pi} \Big(\frac{m^2}{n} +(A+1)n
		\\
		&\qquad  -\frac{1}{2} (\p_x \phi)^2
		- \frac{\mb^2}{\nb}- (A+1)\nb \Big) dx  \Big\rdbrack  dt.
	\end{aligned}
\end{equation}
\begin{proof}[Proof of Lemma \ref{Lem-shift}]
	Since $\sigma'>0$, definitions of $\mathscr{A}_1$ and $\mathscr{A}_2$, \eqref{def-A1} and \eqref{def-A2}, directly imply that
	\begin{align*}
		&\abs{\mathscr{A}_1'(\X_0)-1} \leq \frac{ \norm{\rho_0^\pm}_{L^\infty(\R)} }{\jump{\nb}} \norm{\sigma'}_{L^1(\R)} \leq \frac{ \norm{\rho_0^\pm}_{H^1(\R)} }{\jump{\nb}} \leq \gamma_0,\\
		&\abs{\mathscr{A}_2'(\Y_0)-1} \leq \frac{ \norm{w_0^\pm}_{H^1(\R)} }{\jump{\nb}} \leq \gamma_0.
	\end{align*}
	Hence, there exists $\gamma_0>0$ such that if \eqref{small-periodic} holds, then it follows from the Implicit Function Theorem that $\left(\X_0, \Y_0\right)$ is uniquely determined by \eqref{def-shift0-1} and \eqref{def-shift0-2} provided $\nu_0$ is small enough.
	
	Moreover, when $ \nu=\norm{[\rho_0^\pm, w_0^\pm] }_{H^3((0,\pi^\pm))} $ is small, thanks to Lemma \ref{Lem-periodic}, the existence, uniqueness and regularities of $\X,\Y$ can be easily derived from the ODEs, \eqref{ode-shift}. Besides, the exponential decay rates of $ \X' $ and $ \Y' $ can follow from Lemmas \ref{RH}, \ref{decay-per} directly.
	
	\eqref{X-inf} or \eqref{Y-inf} can be obtained by the method of integrating \subeqref{eqn-pert-shock}{1} or \subeqref{eqn-pert-shock}{2} by parts over an special domain depending on $\X$ or $\Y$, respectively, similarly as the periodic perturbations of viscous shock profiles for the scalar conservation law and the isentropic Navier-Stokes equations. We refer the readers to (\cite{Xin.Y.Y2021}) and (\cite{Huang.Y2021}) for details.
\end{proof}

\subsubsection{Main result for the shock profile}
With the ansatzes defined above, the theorem for the stability of the shock profile under periodic perturbations can be rephrased as follows.
\begin{theorem}\label{thm-shock}
	Assume that \eqref{RH} and \eqref{entropy} hold.
	Then there exist $ \delta_0>0$, $\gamma_0>0$ and $ \varepsilon_0>0$ such that if the strength of the shock profile satisfy \eqref{small-strength}, 
    the periodic perturbations $ \left(\rho_0^\pm, w_0^\pm \right)$ on $(0,\pi^\pm)$ are of zero average and small, i.e. \eqref{average0} and \eqref{small-periodic},  
    and the initial data satisfies \eqref{init-shock}, zero-mass type condition \eqref{zero mass}, and 
	$$  \norm{[\Phi_0,\Psi_0]}_{H^2} + \nu < \varepsilon_0, $$ 
	where $\Phi_0,\Psi_0$ are defined in \eqref{def-anti}, 
	then the problem \eqref{NSP-Ion} and \eqref{initial} admits a unique global solution $ [n,m,\phi] $ satisfying
	\begin{align*}
		n-n^\sha &\in C\left(0,+\infty;  H^1(\R) \right) \cap L^2\left(0,+\infty; H^1(\R) \right),\\
		m-m^\sha &\in C\left(0,+\infty;  H^1(\R) \right) \cap L^2\left(0,+\infty; H^2(\R) \right),\\
		 \phi-\phi^\sha & \in C\left(0,+\infty;  H^2(\R) \right) \cap L^2\left(0,+\infty; H^2(\R) \right).
	\end{align*}
	Moreover, it holds that
	\begin{equation}\label{asym-shock}
		\norm{[n,m,\phi](\cdot,t) - [n^s,m^s,\phi^s](\cdot-st-\X_\infty) }_{L^\infty(\R)} \rightarrow 0 \quad \text{ as } t\rightarrow +\infty,
	\end{equation}
    where $\X_\infty(=\Y_\infty)$ is the constant given by \eqref{X-inf}.
\end{theorem}

\begin{remark}
	When the periodic perturbations vanish, the zero-mass type condition \eqref{zero mass} reduces to 
		\begin{align*}
			s\Big\{
			\int_\R(n_0-n^s(\cdot,0)) dx \Big\}
			= \int_\R(m_0-m^s(\cdot,0)) dx,
		\end{align*}
	and the equations of initial shifts $\X_0, \Y_0$ reduce to 
	\begin{align*}
		\X_0= -\frac{1}{\jump{\nb}} \int_\R(n_0-n^s(\cdot,0)) dx, \quad \Y_0= -\frac{1}{\jump{\mb}} \int_\R(m_0-m^s(\cdot,0)) dx.
	\end{align*}
	They are equivalent to that there exists an initial shift $\X_0$ such that 
	\begin{align*}
			\int_\R [n_0(x)-n^s(x-\X_0),m_0(x)-m^s(x-\X_0)] dx =0,
	\end{align*}
    which is exactly the zero-mass condition imposed in \cite{duan2020a}.
\end{remark}

\subsection{Ansatz and main result for rarefaction waves}
Without confusions, we still use the notations $\sigma,n^\sha, u^\sha$, etc., when studying the problem of the stability of rarefaction waves, but they are \emph{different} from the ones in the problem of the stability of shock profiles.
Recall that $[n^r,u^r,\phi^r]$ are approximate rarefaction waves defined by \eqref{rare-sol} and \eqref{wr}.

Since we can decompose $n^r$, $u^r$ into 
\begin{equation}
	n^r=\nb^-(1-\sigma)+\nb^+\sigma,\quad u^r=\ub^-(1-\eta)+\ub^+\eta,
\end{equation}
with
\begin{equation}
	\sigma=\dfrac{n^r-\nb^-}{\nb^+-\nb^-},\quad\eta=\dfrac{u^r-\ub^-}{\ub^+-\ub^-},
\end{equation} 
it is nature to define the ansatz as
\begin{equation}
	\label{def-ansatz-rare}
	\begin{aligned}
		n^\sha&:=n^-(1-\sigma)+n^+\sigma=n^r+\rho^-(1-\sigma)+\rho^+\sigma,\\ u^\sha&:=u^-(1-\eta)+u^+\eta=u^r+v^-(1-\eta)+v^+\eta,\\
		\phi^\sha&:=\phi^r+\varphi^-(1-\sigma)+\varphi^+\sigma,
	\end{aligned}
\end{equation}
where $[\rho^\pm,w^\pm,v^\pm,\varphi^\pm]$ are defined in \eqref{def-pert}.
It should be remarked that $\phi^r=-\ln n^r$, and  $\phi^r=\phi^-(1-\sigma)+\phi^+\sigma$ does not hold. 

Define the perturbation $(\nt,\ut,\phit)=(n-n^\sha,u-u^\sha,\phi-\phi^\sha)$. With the aid of Lemma  \ref{lem-error}, $(\nt,\ut,\phit)$ satisfies
\begin{equation}
	\label{eqn-pert-rare}
	\begin{cases}
		\p_t\nt+\p_x(nu)-\p_x(n^\sha u^\sha)=k_1,\\
		\p_t\ut+u\p_xu-u^\sha\p_xu^\sha+A\left(\dfrac{\p_xn}{n}-\dfrac{\p_xn^\sha}{n^\sha}\right)-\p_x\phit
		=\dfrac{\p_x^2 u}{n}-\dfrac{\p_x^2u^\sha}{n^\sha}+\dfrac{\p_x^2u^r}{n^r}+k_2,\\
		\p_x^2\phit=\nt+e^{-\phi^\sha}(1-e^{-\phit})-\p_x^2\phi^r+k_3,
	\end{cases}
\end{equation} 
with the initial data
\begin{equation}
	\label{init-rare2}
	[\nt,\ut](x,0)=[\nt_0,\ut_0](x)=[n_0(x)-n^\sha(x,0), u_0(x)-u^\sha(x,0)], \quad \mt_0=\nt\ut,
\end{equation}
and the errors terms $k_i$ satisfy 
\begin{equation}
	\label{eqn-error-rare}
	k_i=\sR \quad \text{for}~i=1,2,3.
\end{equation}

The details to derive \eqref{eqn-pert-rare} and \eqref{eqn-error-rare} are similar to the derivation of \eqref{eqn-pert-shock}-\eqref{eqn-error-shock} given in the appendix Section \ref{subsect-error}. We omit here for simplicity. Moreover, thanks to Lemma \ref{lem-error}, good-decaying terms in $k_1$, $k_2$, and $k_3$ are quadratic in the form $\sR=\sR_{\rm Ti} \sR_{\rm Sp}$. For the stability of rarefaction waves, $\sR_{\rm Ti}$ are the periodic functions and their derivatives, and $\sR_{\rm Sp}$ are the derivatives of the background rarefaction waves. In particular, such forms are kept after derivations for several times due to Lemma \ref{Lem-periodic}, \ref{lem-rare-prop}. 

With the ansatz defined above, the theorem for the long-time behaviors of the rarefaction waves under periodic perturbations can be rephrased as follows.

\begin{theorem}\label{thm-rare}
	Assume that $[\nb^+,\ub^+] \in R_2(\nb^-,\ub^-)$.
	Then there exist $\nu_0>0$ and $ \varepsilon_0>0$ such that if
	the periodic perturbations $ \left(\rho_0^\pm, w_0^\pm \right)$ on $(0,\pi^\pm)$ are of zero average, {\rm i.e.}, \eqref{average0} and 
	\begin{equation}
		\label{small-periodic-rare}
		\nu:=\norm{[\rho_0^\pm, w_0^\pm]}_{H^3((0,\pi^\pm))}\leq \nu_0,
	\end{equation}
	and the initial data satisfies 
	$$  \norm{[\nt_0,\mt_0]}_{H^1} +
	\epsilon < \varepsilon_0, $$ 
	where $\nt_0,\mt_0$ are defined in \eqref{init-rare2}, $\epsilon>0$ is the parameter in \eqref{wr}, then the problem \eqref{NSP-Ion} and \eqref{initial} admits a unique global solution $ [n,m,\phi] $ satisfying
	\begin{align*}
		n-n^\sha &\in C\left(0,+\infty;  H^1(\R) \right) \cap L^2\left(0,+\infty; H^1(\R) \right), \\
		m-m^\sha &\in C\left(0,+\infty;  H^1(\R) \right) \cap L^2\left(0,+\infty; H^2(\R) \right), \\
		\phi-\phi^\sha & \in C\left(0,+\infty;  H^2(\R) \right) \cap L^2\left(0,+\infty; H^2(\R) \right).
	\end{align*}
	Moreover, it holds that
	\begin{equation}\label{asym-rare}
		\norm{[n,m,\phi](\cdot,t) - [n^R,m^R,\phi^R](\cdot/t) }_{L^\infty(\R)} \rightarrow 0 \quad \text{ as } t\rightarrow +\infty.
	\end{equation}
\end{theorem}

\section{Stability of shock profiles}

The proof Theorem \ref{thm-shock} is based on the \textit{a priori} estimates on the solutions to the equations of anti-derivatives $[\Phi,\Psi]$ and $\phit$, since the local existence of $[\Phi,\Psi,\phit]$ follows from the standard iteration method, which is omitted here for simplicity.

First, we rewrite the perturbation equations in terms of anti-derivatives $[\Phi,\Psi]$ and $\phit$.
\subsection{Equations}
In view of \eqref{eqn-pert-shock} and \eqref{def-anti}, we have the following perturbation equations in terms of  the anti-derivatives $[\Phi,\Psi]$ and $\phit$:
\begin{equation}
	\label{eqn-anti-shock}
	\begin{cases}
	\p_t \Phi+\p_x \Psi=H_1,  \\
	\p_t\Psi+ 2 \usinf\p_x \Psi + (A+1-\abs{\usinf}^2 ) \p_x \Phi- \p_x  \left(\dfrac{1}{\nsinf}(\p_x \Psi- \usinf\p_x \Phi) \right)\\
	\qquad
	-\left(\p_x\phi^s_{\infty}\p_x\phit+\p_x^2\phit \right)=H_2+q_1+q_2+\p_x (q_3+q_4),\\
	\p_x^2\phit-\nt - e^{-\phi^s_{\infty}}\phit=h_3+q_5+q_6,
	\end{cases}
\end{equation} 	
where
$$H_i(t,x)=\int_{-\infty}^x h_i(t,y) \,d y , \quad \text{for~}i=1,2$$
\begin{align*}
	 q_1&=-2\left(u^{\sha}-\usinf\right) \mt+\left((u^{\sha})^2-(\usinf)^2 \right) \nt 
	+\left(\p_x \phi^{\sha}-\p_x \phi^s_{\infty}\right) \p_x \phit ,\\
	q_2&=-\left(\frac{m^2}{n}-\frac{\left(m^\sha\right)^2}{n^\sha}-2 u^{\sha} \tilde{m}	+\left(u^{\sha}\right)^2 \tilde{n}\right)
	+\left(\frac{1}{2}\left(\p_x \phi\right)^2-\frac{1}{2}\left(\p_x \phi^\sha\right)^2- \p_x \phi^\sha \p_x\phit \right) ,\\
	q_3& =\frac{\tilde{m}}{n^{\sha}}-\frac{\tilde{m}}{\nsinf}-\frac{u^\sha}{n^\sha} \tilde{n}+\frac{u^s_\infty}{n^s_\infty} \tilde{n}
	,\qquad\qquad
    q_4	 =\frac{m}{n}-\frac{m^{\sha}}{n^{\sha}}-\dfrac{\mt}{n^{\sha}}+\dfrac{u^{\sha}}{n^\sha} \tilde{n}, \\
	q_5& =\left(e^{-\phi^\sha}-e^{-\phi^s_{\infty}}\right) \phit , \qquad\qquad\qquad\quad
	q_6 =-\left(e^{-\phi}-e^{-\phi^{\sha}}+e^{-\phi^{\sha}} \phit\right)  .
\end{align*}
The initial data are 
\begin{equation}
	\label{init-anti}
	[\Phi, \Psi](x,t)=[\Phi_0, \Psi_0].
\end{equation}
It follows from \eqref{def-ansatz-shock}, Lemmas \ref{Lem-periodic}, \ref{Lem-shift} and \ref{lem-shock-prop} that
\begin{align*}
	&\norm{[n^\sha,m^\sha,\phi^\sha]-[\nsinf,\msinf,\phisinf]}_{W^{1,\infty}}\\
	&\lesssim\norm{[n^s_\X,m^s_\Y,\phi^s_\X]-[\nsinf,\msinf,\phisinf]}_{W^{1,\infty}}+\norm{[\rho^\pm,w^\pm,\varphi^\pm]}_{W^{1,\infty}}\nonumber\\
	&\lesssim\delta(|\X(t)-\X_\infty|+|\Y(t)-\Y_\infty|)+\nu e^{-\alpha t}\lesssim\nu e^{-\alpha t}.
\end{align*}
If $\nu$ is sufficiently small, then 
\begin{align}\label{nsha-bound}
	\frac{1}{2}\nb^+\lesssim\nsinf-\nu e^{-\alpha t}\lesssim n^\sha\lesssim\nsinf+\nu e^{-\alpha t}<2\nb^-.
\end{align} 
Thus, we also have
\begin{align}\label{usha-u-bdd}
	\norm{u^\sha-\usinf}_{W^{1,\infty}}\lesssim\norm{n^\sha-\nsinf}_{W^{1,\infty}}+\norm{m^\sha-\msinf}_{W^{1,\infty}}\lesssim\nu e^{-\alpha t}.
\end{align}
Therefore, $q_1,\cdots, q_6$ have the following properties, which is useful in the \textit{a priori} estimates,
\begin{align}\label{qi-est}
	q_1& 
	=\O(1) \nu e^{-\alpha t}\left(\abs{\mt}+\abs{\nt}+\abs{\p_x\phit}\right)  , \\
	q_2& =\O(1)\left(\abs{\mt}^2+\abs{\nt}^2+\abs{\p_x\phit}^2\right) , \\
	q_3&=\O(1) \nu e^{-\alpha t}\left(\abs{\mt}+\abs{\nt}\right) , \\
	q_4&=\O(1)\left(\abs{\mt}^2+\abs{\nt}^2\right) ,\\
	q_5&=\O(1) \nu e^{-\alpha t}\abs{\phit},\\
	q_6&=\O(1)\abs{\phit}^2.
\end{align} 
\subsection{\textit{A priori} estimates} 

In the rest of the paper, we use $C$ to represent the generic positive constant, and ``$f\lesssim g$'' to represent ``$f\leq C g$ '' for some constant $C>0$.

\begin{proposition}[\textit{A priori} estimates]
	\label{prop-apriori-shock}
	For any $T>0$, assume that $[\Phi, \Psi, \phit]$ is a smooth solution to \eqref{eqn-anti-shock} and \eqref{init-anti}. Then there exist positive constants $\delta_0, \gamma_0, \varepsilon_0$ independent of $T$ such that if $\delta<\delta_0, \nu<\gamma_0 \delta$, and 
	\begin{equation}
		\label{eqn-assum-shock}
		\varepsilon:=\sup\limits_{t\in[0,T]} \norm{[\Phi, \Psi, \phit]}_{H^2} < \varepsilon_0,
	\end{equation}
then we have 
\begin{equation}
	\label{eqn-aprioi-shock}
	\begin{aligned}
		\sup\limits_{t\in[0,T]} \norm{[\Phi, \Psi, \phit]}_{H^2}^2+
		&\int_{0}^{T} \Big(\norm{ \abs{\p_x\usinf }^{1/2}\Psi }^2 + \norm{[\Phi_x , \phit_t]}_{H^1}^2 +\norm{[\Psi_x ,  \phit]}_{H^2}^2 \Big) \, dt\\
		&\qquad\qquad
		 \lesssim  \norm{[\Phi_0, \Psi_0]}_{H^2}^2 + \nu.
	\end{aligned}
\end{equation}
\end{proposition}
Now we define 
\begin{equation}
	\label{def-L}
	L:=A+1-\abs{\usinf} ^2.
\end{equation}
It follows from \eqref{def-L} and \eqref{eqn-assum-shock} that
\begin{equation}\label{eqn-L}
	1<L<A+1, \quad \p_t L =2s \usinf \p_x \usinf, \quad \p_x L =-2 \usinf \p_x \usinf.
\end{equation}
The proof of Proposition \ref{prop-apriori-shock} consists of the following series of lemmas.

The following estimate of $h_1,h_2,h_3$ and $H_1,H_2$ are useful in the \textit{a priori} estimates.
\begin{lemma}\label{h-estimates}
	Under the assumptions of Proposition \ref{prop-apriori-shock}, there exist positive constants $\delta_0, \gamma_0, \varepsilon_0$ independent of $T$ such that if $\delta<\delta_0, \nu<\gamma_0 \delta$, and $\varepsilon <\varepsilon_0$, then
	\begin{align}
		\norm{h_1,h_2,h_3}_{H^2}&\lesssim \nu\delta^\frac{1}{2}e^{-\alpha t}, \label{h-esti-eq1}\\
		\norm{H_1,H_2}&\lesssim\nu\delta^{-\frac{1}{2}}e^{-\alpha t}. \label{h-esti-eq2}
	\end{align}
\end{lemma}
\begin{proof}
	For $h_1$, we can write it as 
	\begin{equation}
		h_1=\underbrace{\left\{ - \jump{n}(s+ \X') + \jump{m} \right\} \p_x\sigma_{\X}}_{h_{1,1}} + \underbrace{\p_x (m^s_{\Y}-m^s_{\X})}_{h_{1,2}} + \underbrace{\p_x \sR}_{h_{1,3}}.
	\end{equation}
	It is suffices to estimates $h_{1,1,}$ and $h_{1,2}$. It follows from \subeqref{RH}{1} that
	\begin{equation*}
		h_{1,1}=(-\jump{\nb}\X'-\jump{\rho}(s+\X')+\jump{w})\p_x\sigma_{\X}.
	\end{equation*}
By using Lemma \ref{Lem-periodic} and Lemma \ref{Lem-shift}, one has
\begin{equation*}
	\norm{\jump{\nb}\X'}_{W^{2,\infty}}+\norm{\jump{\rho}(s+\X')}_{W^{2,\infty}}+\norm{\jump{w}}_{W^{2,\infty}}\lesssim\nu e^{-\alpha t}.
\end{equation*}
By using Lemma \ref{lem-shock-prop}, one has
	\begin{align*}
		\norm{\p_x^k\sigma_{\X}}\lesssim\delta^{-1}\norm{\p_\xi^k n^s}\lesssim \delta^{-1}\delta^{k+1}\|e^{-\theta\delta \abs{x}}\|\lesssim\delta^{k-\frac{1}{2}} \quad \text{for~} k=1,2,3.
	\end{align*}
Thus, we have
\begin{equation*}
	\norm{h_{1,1}}_{H^2}\lesssim\nu \delta^\frac{1}{2} e^{-\alpha t}.
\end{equation*}
Similarly, by using Lemma \ref{Lem-shift} and Lemma \ref{lem-shock-prop}, we have
\begin{equation*}
	\norm{h_{1,2}}_{H^2}\lesssim\norm{\p_\xi^2m^s}_{H^2}(\norm{\X-\X_\infty}_{L^\infty}+\norm{\Y-\Y_\infty}_{L^\infty})\lesssim \nu\delta^\frac{3}{2} e^{-\alpha t}.
\end{equation*}
As stated above, for the problem of the stability of shock profiles, $\sR$ in $h_i$ $(i=1,2,3)$ are in the form of $\sR=\sR_{\rm Ti} \sR_{\rm Sp}$, and $\sR_{\rm Ti}$ are the periodic functions and their derivatives, and $\sR_{\rm Sp}$ are the derivatives of the background shock profiles. Therefore, 
\begin{align*}
	\norm{h_{1,3}}_{H^2}
	&\lesssim \norm{\p_x (\sR_{\rm Ti} \sR_{\rm Sp})}_{H^2} 
	\lesssim (\norm{\p_x^3 \sR_{\rm Ti} }_{L^\infty} \norm{\sR_{\rm Sp}}+\cdots + \norm{ \sR_{\rm Ti} }_{L^\infty} \norm{\p_x^3\sR_{\rm Sp}} ) \\
	&\lesssim \nu e^{-2\alpha t} \sum_{k=1}^{4} \norm{\p_x^k\sigma_{\X}} \lesssim \nu\delta^\frac{1}{2} e^{-\alpha t}.
\end{align*}
For the estimate of $H_1$, it follows from Lemma \ref{Lem-shift} and Lemma \ref{lem-shock-prop} that
\begin{equation*}
	\norm{\int_{-\infty}^{x}h_{1,2} dy}
	\lesssim\norm{m_\X^s-m_\Y^s}
	\lesssim\norm{\p_xm^s}(\norm{\X-\X_\infty}_{L^\infty}+\norm{\Y-\Y_\infty}_{L^\infty})
	\lesssim \nu\delta^{\frac{1}{2}} e^{-\alpha t}.
\end{equation*}
It is obvious to yield that
\begin{equation*}
	\norm{\int_{-\infty}^{x}h_{1,3} dy}\lesssim\norm{\sR}
	\lesssim \norm{\sR_{\rm Ti} \sR_{\rm Sp}}  
		\lesssim \nu e^{-2\alpha t} \norm{\p_x\sigma_{\X}}
	\lesssim\nu\delta^{\frac{1}{2}} e^{-\alpha t}.
\end{equation*}
Now we estimate $H_{1,1}:=\int_{-\infty}^{x}h_{1,1}(y)dy$. For $x<st+\X$,
we compute
\begin{align*}
	H_{1,1}=&-\jump{\nb}\X'\int^{x}_{-\infty}\sigma'(y-st-\X)dy-(s+\X')\int^{x}_{-\infty}\jump{\rho}\sigma'(y-st-\X)dy\\
	&+\int^{x}_{-\infty}\jump{w}\sigma'(y-st-\X)dy.
\end{align*}
Noting that $\sigma'(\xi)>0$, we have
	\begin{align*}
		\int_{-\infty}^{st+\X} \abs{H_{1,1}}^2 dx
		&\lesssim \int_{-\infty}^{st+\X} \abs{ \nu e^{-2\alpha t} \int^{x}_{-\infty}\sigma'(y-st-\X)dy }^2 dx\\
		& = \int_{-\infty}^{st+\X} \abs{ \nu e^{-2\alpha t} \jump{\nb}^{-1} (n^s(x-st-\X)-\nb^-) }^2 dx\\
		&\lesssim \nu^2 e^{-4\alpha t}  \int_{-\infty}^{0} e^{-2\theta\delta \abs{\xi}}   d\xi
		\lesssim  \nu^2\delta^{-1}e^{-2\alpha t}.
	\end{align*}
Thus, we have
\begin{equation*}
	\int_{-\infty}^{st}|H_{1,1}|^2dx\lesssim \nu^2e^{-2\alpha t}\delta^{-1}.
\end{equation*}
For $x>st+\X$, we consider $H_{1,1}$ in the form $H_{1,1}=-\int_{x}^{+\infty} h_{1,1}(y)dy$, since $\int_{-\infty}^{+\infty} h_{1,2} dy=\int_{-\infty}^{+\infty} h_{1,3} dy =0$ and $\X, \Y$ are set such that $\int_{-\infty}^{+\infty} h_1 dy =0$ and thus $\int_{-\infty}^{+\infty} h_{1,1} dy =0$ as well. We can use the same argument to yield that
\begin{equation*}
	\int_{st}^{+\infty}|H_{1,1}|^2dx\lesssim \nu^2\delta^{-1}e^{-2\alpha t}.
\end{equation*}
The estimates of $h_2,H_2$ and $h_3$ follows from the same argument, we omit the details here. 
\end{proof}

\begin{lemma}
	\label{lem-0th-shock}
 Under the assumptions of Proposition \ref{prop-apriori-shock}, there exist positive constants $\delta_0, \gamma_0, \varepsilon_0$ independent of $T$ such that if $\delta<\delta_0, \nu<\gamma_0 \delta$, and $\varepsilon <\varepsilon_0$, then
 \begin{equation}
 	\label{ineq-0th-shock}
 	\begin{aligned}
 		&\sup\limits_{t\in[0,T]} \norm{[\Phi, \Psi, \phit,\phit_x]}^2+
 		\int_{0}^{T} \norm{ [\abs{\p_x\usinf }^{1/2}\Psi, \Psi_x] }^2 \, dt  \\
 		&\qquad\qquad
 		\lesssim  \norm{[\Phi_0, \Psi_0]}^2 + \norm{ \phit_0}_{H^1}^2 +(\delta+\varepsilon+\nu) \int_{0}^{T} \norm{ [\Phi_x,\phit_t,\phit_x,\phit] }^2 \, dt +\nu .
 	\end{aligned}
 \end{equation}
\end{lemma}

\begin{proof}
	Multiplying \subeqref{eqn-anti-shock}{1} by $\Phi$ and \subeqref{eqn-anti-shock}{2} by $\Psi/L$, respectively, summing them up and integrating the resulting equations with respect to $x$ over $\R$ gives that 
	\begin{align}\label{0th-shock-energy}
		& \frac{1}{2} \frac{d}{d t}(\Phi, \Phi)+\frac{1}{2} \frac{d}{d t}\left(\frac{\Psi}{L}, \Psi\right)+(L_1 \Psi, \Psi)+\left(\frac{\Psi_x}{L \nsinf}, \Psi_x\right)-\left(\p_x^2 \phit, \frac{\Psi}{L}\right) \nonumber\\
		& \qquad
		=\underbrace{\left(\frac{\p_x L}{L^2 \nsinf} \Psi,  \Psi_x-\usinf \Phi_x\right)}_{\textbf{I}_1}
		+\underbrace{\left(\frac{\usinf}{L \nsinf} \Psi_x , \Phi_x\right)}_{\textbf{I}_2}
		+\underbrace{\left(\p_x \phisinf \phit_x, \frac{\Psi}{L}\right)}_{\textbf{I}_3} \\
		& \qquad \quad
		+\underbrace{\left(H_1, \Phi\right)+\left(H_2, \frac{\Psi}{L}\right)}_{\textbf{I}_4}
		+\underbrace{\left(q_1+q_2, \frac{\Psi}{L}\right)-\left(q_3+q_4, \p_x\left(\frac{\Psi}{L}\right)\right)}_{\textbf{I}_5}\nonumber,
	\end{align}
	where 
	\begin{align*} 
	 L_1&=-\p_t \left(\frac{1}{2 L}\right)-\p_x\left(\frac{\usinf}{L}\right)
	     =\frac{1}{2 L^2}(\p_t L+2 \usinf\p_x L-2 \p_x \usinf L) \\
		& =-\frac{1}{L^2}(L-s \usinf+2\abs{\usinf}^2) \p_x \usinf \geq \frac{1}{2}\abs{\p_x \usinf} ,
	\end{align*}
when $\delta$ is small. 

The term $-(\p_x^2 \phit, L^{-1}{\Psi}) $ on the left hand side of \eqref{0th-shock-energy} provides the energy of $\phit$. Indeed, we utilize integration by parts and \subeqref{eqn-anti-shock}{1} to get
    \begin{align*}
		-\left(\p_x^2 \phit, \frac{\Psi}{L}\right)
		=\left(\phit_x, \frac{\Psi_x}{L}\right)-\left( \phit_x,\frac{L_x }{L^2}\Psi\right)=\left( \phit_x,-\frac{\Phi_t }{L}\right) +\left(\phit_x, \frac{H_1}{L}\right)
		-\left(\phit_x,\frac{L_x }{L^2}\Psi\right).
	\end{align*}
	Integration by parts again, we have
	    \begin{align*}
		\left( \phit_x,-\frac{\Phi_t }{L}\right)
		=&-\frac{d}{d t}\left(\phit_x, \frac{\Phi}{L}\right)
		+\left(\p_t\phit_x, \frac{\Phi}{L}\right) 
		+\left(\phit_x,-\frac{L_t}{L^2} \Phi\right) \\
		=&-\frac{d}{d t}\left(\phit_x, \frac{\Phi}{L}\right)-\left(\phit_t, \frac{ \Phi_x}{L}\right)+\left(\phit_t, \frac{L_x }{L^2}\Phi\right)
		+\left(\phit_x,-\frac{L_t}{L^2} \Phi\right) \\
		=&-\frac{d}{d t}\left(\phit_x, \frac{\Phi}{L}\right)-\left(\phit_t, \frac{ \Phi_x}{L}\right)+\frac{d}{dt}\left(\phit,\frac{L_x}{L^2}\Phi\right)-\left(\phit,\frac{L_x}{L^2}\Phi_t\right)+\left(\phit,\frac{L_t}{L^2}\Phi_x\right).
	\end{align*}
	It follows from $\Phi_x=\nt$ and \subeqref{eqn-anti-shock}{3} that
	    \begin{align*}
		-\left(\phit_t, \frac{ \Phi_x}{L}\right)
		=&-\left(\phit_t, L^{-1}(\p_x^2 \phit-e^{-\phisinf } \phit)\right)
		+\left(\phit_t, L^{-1}(q_5+q_6)\right) +\left(\phit_t, L^{-1}h_3\right)\\
		=& 
		\frac{1}{2} \frac{d}{dt}\left(\phit_x,\frac{\phit_x}{L} \right) 
		+\frac{1}{2} \frac{d}{dt} \left( \phit,  \frac{e^{-\phisinf}\phit}{L} \right)
		-\left(\phit_t, \frac{L_x}{L^2} \phit_x  \right)+\frac{1}{2} \left(\phit_x , \frac{L_t}{L^2} \phit_x \right) \\
		&
		-\frac{1}{2} \left(\phit, \p_t (L^{-1} e^{-\phisinf}) \phit \right)
		+\left(\phit_t, L^{-1}(q_5+q_6)\right) +\left(\phit_t, L^{-1}h_3\right).
	\end{align*}
Combining the above equations yields 
\begin{align*}
	-\left(\p_x^2 \phit, \frac{\Psi}{L}\right)
	=&\frac{1}{2} \frac{d}{dt}\left(\p_x \phit,\frac{\p_x \phit}{L} \right) 
	+\frac{1}{2} \frac{d}{dt} \left( \phit,  \frac{e^{-\phisinf}\phit}{L} \right) -\frac{d}{d t}\left(\p_x \phit, \frac{\Phi}{L}\right)+\frac{d}{dt}\left(\phit,\frac{L_x}{L^2}\Phi\right) \\
	&\underbrace{-\left(\phit_t, \frac{L_x}{L^2} \phit_x  \right) +\frac{1}{2} \left(\phit_x , \frac{L_t}{L^2} \phit_x \right)
	-\frac{1}{2} \left(\phit, \p_t (L^{-1} e^{-\phisinf}) \phit \right)}_{-\textbf{I}_6}\\
	&\underbrace{-\left(\phit,\frac{L_x}{L^2}\Phi_t\right)+\left(\phit,\frac{L_t}{L^2}\Phi_x\right)}_{-\textbf{I}_7}+\underbrace{\left(\phit_x, \frac{H_1}{L}\right)}_{-\textbf{I}_8}
	+\underbrace{\left(\phit_t, L^{-1}(q_5+q_6)\right)}_{-\textbf{I}_9}+\underbrace{\left(\phit_t,L^{-1}h_3\right)}_{-\textbf{I}_{10}}.
\end{align*}
Plugging it into \eqref{0th-shock-energy}, we have
\begin{align}\label{0th-shock-energy-1}
\frac{1}{2} \frac{d}{d t}&\left\{(\Phi, \Phi)+\left( \phit_x,\frac{\phit_x }{L} \right) 
-2\left(\phit_x, \frac{\Phi}{L}\right)+\left( \phit,  \frac{e^{-\phisinf}\phit}{L} \right)+\left(\frac{\Psi}{L}, \Psi\right)+2\left(\phit,\frac{L_x}{L^2}\Phi\right)\right\}\nonumber \\
&+(L_1 \Psi, \Psi)+\left(\frac{\Psi_x}{L \nsinf},  \Psi_x\right)=\textbf{I}_1+\cdots+\textbf{I}_9.
\end{align} 
Now we estimate terms $\textbf{I}_1, \cdots, \textbf{I}_9$ one by one. It follows from \eqref{eqn-L}, \eqref{eqn-shock-deriv} and the Young inequality that
\begin{align*}
	\abs{\textbf{I}_1}\lesssim \Big( \abs{\p_x \usinf} \abs{\Psi}, \abs{\p_x \Phi} + \abs{\p_x \Psi}\Big)\lesssim \delta \norm{\abs{\p_x \usinf}^{1/2} \Psi}^2 +\delta \norm{[\Phi_x, \Psi_x] }^2.
\end{align*}
By using the Galilean invariant of NSP system, \eqref{RH} and \eqref{entropy}, we can assume without loss of generality that $\ub^-=-\ub^+>0$ so that the Rankine-Hugoniot condition \subeqref{RH}{1} implies
\[|\ub^-|=|\ub^+|\lesssim|\nb^+-\nb^-|=\delta.\]
 Then, it holds that
 \begin{equation}\label{usinf-bdd}
 	|u^s_\infty|\leq|\ub^{\pm}|\lesssim\delta.
 \end{equation}
Thus, we have
\begin{align*}
		\abs{\textbf{I}_2}\lesssim \delta \norm{[\Phi_x, \Psi_x] }^2.  
\end{align*}
It follows from \eqref{eqn-shock}, \eqref{eqn-shock-deriv} and the Young inequality that
\begin{align*}
	\abs{\textbf{I}_3}\lesssim\left|\left(\frac{(n^s_\infty)^2 \p_xu^s_\infty}{\nb^-|\ub^--s|}\tilde{\phi}_x,\frac{\Psi}{L}\right)\right|\lesssim \delta \norm{\abs{\p_x \usinf}^{1/2} \Psi}^2 +\delta \norm{\phit_x }^2  .
\end{align*}
By using the estimate of $H_1,H_2$, Lemma \ref{h-estimates}, one gets
\begin{align*}
	\abs{\textbf{I}_4} +\abs{{\textbf{I}_8}}\lesssim \norm{[H_1,H_2]}\norm{[\Phi,\Psi,\phit_x]} 
	\lesssim \nu \delta^{-1/2}e^{-\alpha t} \norm{[\Phi,\Psi,\phit_x]}\lesssim \gamma_0 e^{-\alpha t} \norm{[\Phi,\Psi,\phit_x]}^2 +\nu e^{-\alpha t}.
\end{align*}
It follows from estimates of $q_1,\cdots,q_4$ and \eqref{eqn-L} that
\begin{align*}
	\abs{\textbf{I}_5} 
	&\lesssim\norm{\Psi}\|q_1\|+\|\Psi\|_{L^\infty}\|q_2\|_{L^1}+\norm{\p_x\left(\frac{\Psi}{L}\right)}\|q_3\|_{L^2}+\norm{\p_x\left(\frac{\Psi}{L}\right)}_{L^\infty}\|q_4\|_{L^1}\\
	&\lesssim \nu e^{-\alpha t}(\norm{\Psi}+\norm{\Psi_x})\norm{[\mt,\nt,\phit_x]}+\norm{\Psi}_{W^{1,\infty}}\norm{[\mt,\nt,\phit_x]}^2\\
	&\lesssim \nu e^{-\alpha t}\|\Psi\|^2+\nu e^{-\alpha t}\norm{\Psi_x}^2+(\nu e^{-\alpha t}+\varepsilon) \norm{[\mt,\nt,\phit_x]}^2,
\end{align*}
where we have used $\norm{\Psi}_{W^{1,\infty}}\lesssim \varepsilon$ which follows from Sobolev embedding $H^1(\mathbb{R})\hookrightarrow L^\infty(\R)$ and \eqref{eqn-assum-shock}. 

In view of  \eqref{eqn-L} and \eqref{eqn-anti-shock}, one has
\begin{align*}
	\abs{\textbf{I}_6} &\lesssim\norm{L_x}_\infty\norm{\phit_t}\norm{\phit_x}+\norm{L_t}_\infty(\norm{\phit_x}^2+\norm{\phit}^2)\lesssim\delta \norm{[\phit_x, \phit_t, \phit]}^2 ,
\end{align*}
and
\begin{align*}
	\abs{\textbf{I}_7} &\lesssim\norm{\phit}\left(\norm{L_x}_\infty\norm{[\Psi_x,H_1]}+\norm{\phit_t}\|\Phi\|+\norm{L_t}_\infty\norm{\Phi_x}\right)\lesssim\delta \norm{[\phit,\Phi_x,\Psi_x]}^2 +\nu e^{-\alpha t}.
\end{align*}
For $\textbf{I}_9,\textbf{I}_{10}$, we utilize Sobolev embedding, \eqref{eqn-assum-shock} and the Young inequality to get
\begin{align*}
	|\textbf{I}_9|&\lesssim (\nu e^{-\alpha t}+\norm{\phit}_{L^\infty})\norm{\phit_t}\norm{\phit}\lesssim(\nu e^{-\alpha t}+\varepsilon)\norm{[\phit,\phit_{t}]}^2 ,\\
	\abs{\textbf{I}_{10}}&\lesssim\nu e^{-\alpha t}\norm{\phit_t}\lesssim\nu e^{-\alpha t}\left(\norm{\phit_t}^2+1\right).
\end{align*}
Since the discriminant is $L^{-2}(1-L)<0$ for $L>1$ and $|L_x|\lesssim\delta^2$, we have
\begin{align*}
(\Phi,\Phi) +  \left(\phit_x,\frac{\phit_x}{L} \right)  &\lesssim \frac{1}{2} (\Phi,\Phi) + \frac{1}{2} \left(\phit_x,\frac{\phit_x}{L} \right) -  \left(\phit_x,\frac{\Phi}{L} \right),\\
	2\left(\phit,\frac{L_x}{L^2}\Phi\right)&\lesssim\delta^2\left(\phit,\phit\right)+\delta^2\left(\Phi,\Phi\right).
\end{align*} 
Substituting the above estimates of $\textbf{I}_i$ into \eqref{0th-shock-energy-1} and integrating in time $[0,\tau]$ yield that
\begin{align*}
	&\norm{[\Phi,\Psi,\phit,\phit_x]}^2(\tau)+\int_{0}^{\tau}\norm{[|\p_xu^s_\infty|^{1/2}\Psi,\Psi_x]}^2dt\\
	&\lesssim \norm{[\Phi_0,\Psi_0]}^2+\norm{\phit_0}_{H^1}^2+(\nu+\gamma_0) \sup_{t\in[0,\tau]}\norm{[\Phi,\Psi,\phit_x]}^2\\
	&\quad+(\delta+\varepsilon+\nu) \int_{0}^{\tau} \norm{ [\Phi_x,\Psi_x,\phit_t,\phit_x,\phit] }^2 \, dt +\nu,
\end{align*} 
provided $\delta_0$ is sufficiently small. Therefore, by taking supermum with respect to $\tau$ in $[0,T]$ and choosing $\nu_0,\gamma_0$ sufficiently small, we have proved \eqref{ineq-0th-shock}.
\end{proof}

\begin{lemma}
	\label{lem-1st-shock}
	Under the assumptions of Proposition \ref{prop-apriori-shock}, there exist positive constants $\delta_0, \gamma_0, \varepsilon_0$ independent of $T$ such that if $\delta<\delta_0, \nu<\gamma_0 \delta$, and $\varepsilon <\varepsilon_0$, then
	\begin{equation}
		\label{ineq-1st-shock}
		\begin{aligned}
			&\sup\limits_{t\in[0,T]} \norm{\Phi_x}^2+
			\int_{0}^{T} \left(\norm{\Phi_x}^2+ \norm{ \phit }_{H^2}^2 \right) \, dt  \\
			&\qquad\qquad
			\lesssim  \norm{\Psi_0}^2 + \norm{ \Phi_0}_{H^1}^2 +(\delta+\varepsilon+\nu) \int_{0}^{T} \norm{\phit_t}^2 \, dt +\nu .
		\end{aligned}
	\end{equation}
	and 
	\begin{equation}\label{ineq-2nd-shock-phit}
		\norm{\phit(t)}_{H^2}^2\lesssim\norm{\Phi_x(t)}^2+\nu e^{-\alpha t},\quad\forall t\in[0,T].
	\end{equation}
\end{lemma}

\begin{proof}
%
Multiplying \subeqref{eqn-anti-shock}{2} by $\Phi_x$, and integrating the resulting equations with respect to $x$ over $\R$ give that 
\begin{align*}
	&\left(\Psi_t,\Phi_x\right)+(L \Phi_x, \Phi_x)-\left(\p_x \Big( \frac{1}{\nsinf} (\Psi_x-\usinf \Phi_x) \Big),\Phi_x\right)-\left(\p_x^2 \phit, \Phi_x\right)\\
	&\quad=-\left(2\usinf\Psi_x,\Phi_x\right)+\left(\p_x \phisinf \p_x \phit, \Phi_x \right)-\left(\p_x H_2, \Phi\right)+\left(q_1+q_2, \Phi_x\right)-\left(q_3+q_4, \p_x^2 \Phi\right).
\end{align*}
It follows from \subeqref{eqn-anti-shock}{1} that 
\begin{align*}
	\left(\Psi_t,\Phi_x\right)=\frac{d}{dt}\left(\Psi,\Phi_x\right)-\left(\Psi,\p_t\Phi_x\right)=\frac{d}{dt}\left(\Psi,\Phi_x\right)-\left(\Psi,\p_x H_1\right)-\left(\Psi_x,\Psi_x\right)
\end{align*}
and
\begin{align*}
	&-\left(\p_x \Big( \frac{1}{\nsinf} (\Psi_x-\usinf \Phi_x) \Big),\Phi_x\right)=-\left(\frac{\p_x^2\Psi}{\nsinf} ,\Phi_x\right)-\left(\p_x\left(\frac{1}{\nsinf}\right)\Psi_x,\Phi_x\right)-\left(\frac{\usinf}{\nsinf}\Phi_x,\p_x^2\Phi \right)\\
	&=\left(\frac{\p_t\Phi_x}{\nsinf} ,\Phi_x\right)-\left(\frac{\p_xH_1}{\nsinf},\Phi_x\right)-\left(\p_x\left(\frac{1}{\nsinf}\right)\Psi_x,\Phi_x\right)+\left(\frac{1}{2}\p_x\left(\frac{\usinf}{\nsinf}\right)\Phi_x,\Phi_x \right)\\
	&=\frac{1}{2}\frac{d}{dt}\left(\frac{1}{\nsinf}\Phi_x,\Phi_x\right)-\left(\frac{1}{2}\left(\p_t \left(\frac{1}{\nsinf} \right)-\p_x \left(\frac{\usinf}{\nsinf} \right) \right) \Phi_x  , \Phi_x\right)-\left(\frac{\p_xH_1}{\nsinf},\Phi_x\right)-\left(\frac{\p_x\nsinf}{\abs{\nsinf}^2}\Psi_x,\Phi_x\right) .
\end{align*}
Taking the inner product between \subeqref{eqn-anti-shock}{3} and $\p_x^2\phit$ and noticing $\Phi_x=\nt$, we obtain that
\begin{align}\label{2nd-shock-disp-phit}
	(\p_x^2\phit,\p_x^2\phit)-(\Phi_x,\p_x^2\phit_x)+\left(e^{-\phisinf}\phit_x,\phit_x\right)=(e^{-\phisinf}\p_x\phisinf\phit,\phit_x)+(q_5+q_6,\p_x^2\phit).
\end{align}
As a consequence of the above equations, we have
\begin{align}\label{1st-shock-energy}
	& \frac{1}{2} \frac{d}{d t}\left(\frac{1}{\nsinf}\Phi_x, \Phi_x\right)+\frac{1}{2} \frac{d}{d t}\left(\Psi, \Phi_x\right)+(L\Phi_x,\Phi_x)-2\left(\p_x^2\phit,\Phi_x\right)+(\p_x^2\phit,\p_x^2\phit)+\left(e^{-\phisinf}\phit_x,\phit_x\right)\nonumber\\
	& \quad
	=\underbrace{
	\left( (-2\usinf+\frac{\p_x \nsinf}{\abs{\nsinf}^2 }  ) \Psi_x + \frac{1}{2} (\p_t (\frac{1}{\nsinf} )-\p_x (\frac{\usinf}{\nsinf} ) ) \Phi_x  , \Phi_x\right)
	+\left( \Psi_x , \Psi_x\right)
    +\left(\p_x \phisinf\phit_x, \Phi_x \right)}_{\textbf{J}_1}\nonumber \\
	& \quad  \quad
	+\underbrace{\left(\p_x H_1, \Psi+ \frac{1}{\nsinf}\Phi_x\right)-\left(\p_x H_2, \Phi\right)}_{\textbf{J}_2}
	+\underbrace{\left(q_1+q_2,\Phi_x\right)-\left(q_3+q_4, \p_x^2 \Phi\right)}_{\textbf{J}_3} \\
	&\qquad+\underbrace{(e^{-\phisinf}\p_x\phisinf\phit,\phit_x)+(h_3+q_5+q_6,\p_x^2\phit)}_{\textbf{J}_4}.\nonumber
\end{align}
The terms $\textbf{J}_1,\cdots,\textbf{J}_4$ can be bounded as follows. By using \eqref{eqn-shock-deriv} and the Young inequality, one has
\begin{align*}
	\abs{\textbf{J}_1}&\lesssim\delta\norm{[\p_x\Phi,\p_x\phit]}^2+\norm{\p_x\Psi}^2.
\end{align*}
It follows from the estimates of $H_1, H_2$ and the Young inequality that
\begin{align*}
	\abs{\textbf{J}_2}&\lesssim\norm{[h_1,h_2]}(\norm{\Phi}_{H^1}+\norm{\Psi})\lesssim\nu\delta^\frac{1}{2}e^{-\alpha t}(\norm{\Phi}_{H^1}+\norm{\Psi})\lesssim\nu\delta e^{-\alpha t}(\norm{\Phi}_{H^1}^2+\norm{\Psi}^2)+\nu e^{-\alpha t}.
\end{align*}
By using estimates of $q_i$, Sobolev embedding, \eqref{eqn-assum-shock} and the Young inequality, we have
\begin{align*}
	\abs{\textbf{J}_3}&\lesssim\|q_1\|\|\Phi_x\|+\|q_2\|_{L^1}\|\Phi_x\|_{L^\infty}+(\|q_3\|+\|q_4\|)\|\p_x^2\Phi\|\\
	&\lesssim\nu e^{-\alpha t}\norm{[\mt,\nt,\phit_x]}\|\Phi_x\|_{H^1}+\norm{[\mt,\nt,\phit_x]}^2\norm{\Phi_x}_{H^1}+\|[\mt,\nt]\|_{L^\infty}\|[\mt,\nt]\|\|\p_x^2\Phi\|\\
	&\lesssim\varepsilon\left(\nu e^{-\alpha t}+\norm{[\Phi_x,\Psi_x,\phit]}_{H^1}^2\right), 
\end{align*}
and
\begin{align*}
	\abs{\textbf{J}_4}&\lesssim|\p_x\phisinf|_{L^\infty}\norm{\phit}\norm{\phit_x}+\left(\nu e^{-\alpha t}+\nu e^{-\alpha t}\norm{\phit}+\norm{\phit}_{L^\infty}\norm{\phit}\right)\norm{\p_x^2\phit}\\
	&\lesssim\delta^2\norm{\phit}\norm{\phit_x}+\varepsilon\left(\nu e^{-\alpha t}\norm{\phit}+\norm{\phit}_{L^\infty}\norm{\phit}\right)\\
	&\lesssim(\delta+\varepsilon)\norm{\phit}_{H^1}^2+\nu e^{-\alpha t}.
\end{align*}
Noting that
\begin{align*}
	(L\Phi_x,\Phi_x)-2\left(\p_x^2\phit,\Phi_x\right)+(\p_x^2\phit,\p_x^2\phit)\gtrsim \norm{\Phi_x}^2+\norm{\p_x^2\phit}^2\\
	\left(\frac{1}{\nsinf}\Phi_x,\Phi_x\right)+(\Psi,\Phi_x)\gtrsim \norm{\Phi_x}^2-\norm{\Psi},
\end{align*}
it follows from plugging estimates of $J_i$ into \eqref{1st-shock-energy} and integrating with respect to time in $[0,\tau]$ that 
\begin{equation}\label{1st-shock-energy-Phi}
	\begin{aligned}
	&\norm{\Phi_x}(\tau)+\int_{0}^{\tau}\norm{\Phi_x}^2+\norm{\p_x^2\phit}^2dt\\
	&\lesssim \norm{\Phi_{0x}}^2+\|\Psi_0\|^2+ \nu\delta\sup_{t\in[0,\tau]}(\norm{\Phi}^2+\norm{\Psi}^2)+(\delta+\varepsilon)\int_{0}^{\tau}\norm{[\Phi_x,\Psi_x,\phit]}_{H^1}^2dt+\nu. 
\end{aligned}
\end{equation}
Next, taking the $L^2$ inner product between \subeqref{eqn-anti-shock}{3} and $\phit$, and integration by parts yield
\begin{align}\label{1st-shock-energy-phit}
	\norm{\phit}_{H^1}^2\lesssim\norm{\p_x\Phi}^2+\norm{[q_5,q_6]}^2\lesssim\norm{\p_x\Phi}^2+\varepsilon\norm{\phit}_{H^1}^2+\nu e^{-\alpha t} \, .
\end{align}
Multiplying \eqref{1st-shock-energy-phit} with suitable constant and adding to \eqref{1st-shock-energy-Phi}, using \eqref{0th-shock-energy} and choosing $\delta_0$ further sufficiently small, we have proved \eqref{ineq-1st-shock}. In view of \eqref{2nd-shock-disp-phit}, \eqref{1st-shock-energy-phit} and estimate of $J_4$, we can get \eqref{ineq-2nd-shock-phit}.
\end{proof}

\begin{lemma}
	\label{lem-2nd-shock}
Under the assumptions of Proposition \ref{prop-apriori-shock}, there exist positive constants $\delta_0, \gamma_0, \varepsilon_0$ independent of $T$ such that if $\delta<\delta_0, \nu<\gamma_0 \delta$, and $\varepsilon <\varepsilon_0$, then
\begin{equation}
	\label{ineq-2nd-shock}
	\begin{aligned}
		&\sup\limits_{t\in[0,T]} (\norm{\p_x^2 \Phi}^2+\norm{\p_x^2 \Psi}^2)+
		\int_{0}^{T} \left(\norm{\p_x \Phi}_{H^1}^2+ \norm{ \phit }_{H^2}^2+\norm{ \p_t \phit}_{H^1}^2\right) \, dt
		\lesssim  \norm{[\Phi_0,\Psi_0]}_{H^2}^2+\nu.
	\end{aligned}
\end{equation}	
\end{lemma}
\begin{proof}
	To derive higher order energy estimates on $[\Phi,\Psi]$, we first introduce $\ut=u-u^\sha$ and rewrite \eqref{eqn-pert-shock} into
	\begin{equation}
		\label{eqn-pert-shock-new}
		\begin{cases}
			\nt_t+n\ut_x+u\nt_x+ u^\sha_x\nt+n^\sha_x\ut=h_1 \,, \\
			n\ut_t+nu\ut_x+A\nt_x+n\ut u^\sha_x+\nt(u^\sha_t+u^\sha u^\sha_x)-\p_x^2\ut-(n\phi_x-n^\sha\phi^\sha_x)=h_2-h_1u^\sha \,,\\
			\p_x^2\phit-\nt + (e^{-\phi} -e^{-\phi^\sha})=h_3 \,.
		\end{cases}
	\end{equation} 
Set $$\Lambda(n,n^\sha)=A\int^{n}_{n^\sha}\frac{s-n^\sha}{s^2}ds=A\left(\ln\left(\frac{n}{n^\sha}\right)+\frac{n^\sha}{n}-1\right).$$
	Multiplying \subeqref{NSP-Ion}{1} and \subeqref{eqn-pert-shock-new}{2} by $\Lambda(n,n^\sha)$ and $\ut$ respectively, integrating the resulting equation with respect to $x$ over $\R$ give that 
	\begin{align*}
		\frac{d}{dt}\left\{\frac{1}{2}\left(n\ut,\ut\right)+\left(n,\Lambda\right)\right\}+\norm{\ut_x}^2
		=&-\left(n\ut u^\sha_x+\nt(u^\sha_t+u^\sha u^\sha_x),\ut\right)+\left(n\phi_x-n^\sha\phi^\sha_x,\ut\right)\\
		&+\left(h_2-h_1u^\sha,\ut\right) .
	\end{align*}
To estimate terms on the right hand side of the above equation, we first obtain from \eqref{nsha-bound}, \eqref{usha-u-bdd} and \eqref{usinf-bdd} that
\begin{align}
	&\norm{u^\sha}_{L^\infty}\lesssim\norm{\usinf}_{L^\infty}+\nu e^{-\alpha t}\lesssim\delta+\nu e^{-\alpha t},\label{usha-bound}\\
	&\norm{[n^\sha_x,u^\sha_x,\phi^\sha_x]}_{L^\infty}+\norm{\p_tu^\sha}_{L^\infty}\lesssim\delta+\nu e^{-\alpha t}.\label{n-usha-1st-bound}
\end{align}
As a consequence, it follows from the Young inequality and estimates of $h_1,h_2$ that 
\begin{align}\label{u-0th-shock-energy}
	\frac{d}{dt}\left\{\frac{1}{2}\left(n\ut,\ut\right)+\left(n,\Lambda\right)\right\}+\norm{\ut_x}^2&\lesssim(\delta+\nu e^{-\alpha t})\norm{[\ut,\nt]}^2+\norm{\phit_x}^2+\norm{[h_1,h_2]}\norm{\ut},\nonumber\\
	&\lesssim(\delta+\nu e^{-\alpha t})\norm{[\ut,\nt]}^2+\norm{\phit_x}^2+\delta^\frac{1}{2}\varepsilon\nu e^{-\alpha t}.
\end{align}
Similarly, taking inner product between \subeqref{eqn-pert-shock-new}{2} and $-\p_x^2\ut/n$ gives that
\begin{equation*}
	\begin{aligned}
		\frac{1}{2}\frac{d}{dt}\left(\ut_x,\ut_x\right)+\left(\frac{\p_x^2\ut}{n},\p_x^2\ut\right)&=\left(u\ut_x+\frac{A}{n}\nt_x+\ut u^\sha_x+\frac{\nt}{n}(u^\sha_t+u^\sha u^\sha_x),\p_x^2\ut\right)\\
		&\quad-\left(\frac{1}{n}(n\phi_x-n^\sha\phi^\sha_x),\p_x^2\ut\right)-\left(\frac{1}{n}(h_2-h_1u^\sha),\p_x^2\ut\right)\\
		&\lesssim\mu_1\norm{\p_x^2\ut}^2+\frac{1}{\mu_1}(\norm{[\ut_x,\nt_x,\phit_x]}^2+\delta\norm{\ut,\nt}^2+\nu e^{-\alpha t}),
	\end{aligned}
\end{equation*}
which implies that
\begin{align}\label{u-1st-shock-energy}
	\frac{d}{dt}\norm{\ut_x}^2+\norm{\p_x^2\ut}^2\lesssim \norm{[\ut_x,\nt_x,\phit_x]}^2+\delta\norm{\ut,\nt}^2+\nu e^{-\alpha t}.
\end{align}

Next, taking the derivative $\p_x$ on \subeqref{eqn-pert-shock-new}{1} and inner product with $\p_x\nt/n^2$ yield that
\begin{align}\label{2nd-shock-energy-Phi}
	\frac{1}{2}\frac{d}{dt}\left(\frac{\nt_x}{n^2},\nt_x\right)+\left(\frac{1}{n}\p_x^2\ut,\nt_x\right)=&-\left(\frac{\nt_x}{n^2},2n^\sha_x\ut_x+\frac{1}{2}u^\sha_x\nt_x+\p_x^2u^\sha\nt+\p_x^2n^\sha\ut\right)\nonumber\\
	&-\left(\frac{\ut_x}{2n^2}\nt_x,\nt_x\right)+\left(\frac{\nt_x}{n^2},\p_xh_1\right),
\end{align}
where we have used
\begin{align*}
	&\left(-\left(\frac{1}{2n^2}\right)_t\nt_x+\frac{u}{n^2}\p_x^2\nt+\frac{u_x}{n^2}\nt_x,\nt_x\right)=\left(\frac{n_t}{n^3}+\frac{un_x}{n^3}+\frac{u_x}{2n^2},\nt_x\right)\\
	&=-\left(\frac{u_x}{2n^2}\nt_x,\nt_x\right)=-\left(\frac{\ut_x}{2n^2}\nt_x,\nt_x\right)-\left(\frac{u^\sha_x}{2n^2}\nt_x,\nt_x\right),
\end{align*}
and
\begin{align*}
	\left(\frac{1}{n^2}n_x\ut_x,\nt_x\right)=\left(\frac{1}{n^2}\nt_x\ut_x,\nt_x\right)+\left(\frac{1}{n^2}n^\sha_x\ut_x,\nt_x\right).
\end{align*}
For the troublesome term $\left(\frac{1}{n}\p_x^2\ut,\nt_x\right)$ on the left hand side above, we multiply \subeqref{eqn-pert-shock-new}{2} with $\nt_x/n$ to get
\begin{align*}
	\left(\frac{1}{n}\p_x^2\ut,\p_x\nt\right)=&A\left(\frac{\p_x\nt}{n},\p_x\nt\right)+\frac{d}{dt}\left(\ut,\nt_x\right)-(\ut_x,n\ut_x+u\nt_x+ u^\sha_x\nt+n^\sha_x\ut-h_1)\\
	&+\left(u\ut_x+\ut u^\sha_x,\nt_x\right)+\left(\nt(u_t^\sha+u^\sha u^\sha_x),\frac{\nt_x}{n}\right)\\
	&-\left(n\phi_x-n^\sha\phi^\sha_x,\frac{\nt_x}{n}\right)-\left(h_2-h_1u^\sha,\frac{\nt_x}{n}\right),
\end{align*}
where we have used the following equations from \subeqref{eqn-pert-shock-new}{1} and integration by parts 
\begin{align*}
	\left(\ut_t,\nt_x\right)=\frac{d}{dt}\left(\ut,\nt_x\right)-(\ut_x,n\ut_x+u\nt_x+ u^\sha_x\nt+n^\sha_x\ut-h_1).
\end{align*}
Substituting this into \eqref{2nd-shock-energy-Phi} to get
\begin{align*}
	&\frac{1}{2}\frac{d}{dt}\left(\frac{\nt_x}{n^2},\nt_x\right)+\frac{d}{dt}\left(\ut,\nt_x\right)+A\left(\frac{\p_x\nt}{n},\p_x\nt\right)\\
   &=\underbrace{(\ut_x,n\ut_x)}_{\textbf{K}_1}+\underbrace{(\ut_x, u^\sha_x\nt+n^\sha_x\ut)-\left(\ut u^\sha_x,\nt_x\right)-\left(\nt(u_t^\sha+u^\sha u^\sha_x),\frac{\nt_x}{n}\right)}_{\textbf{K}_2}\\
	&\quad+\underbrace{\left(n\phi_x-n^\sha\phi^\sha_x,\frac{\nt_x}{n}\right)}_{\textbf{K}_3}-\underbrace{\left(\frac{\nt_x}{n^2},2n^\sha_x\ut_x+\frac{1}{2}u^\sha_x\nt_x+\p_x^2u^\sha\nt+\p_x^2n^\sha\ut\right)}_{\textbf{K}_4}\\
	&\quad-\underbrace{\left(\frac{\ut_x}{2n^2}\nt_x,\nt_x\right)}_{\textbf{K}_5}+\underbrace{\left(h_2-h_1u^\sha,\frac{\nt_x}{n}\right)+\left(\frac{\nt_x}{n^2},\p_xh_1\right)-\left(\ut_x,h_1\right)}_{\textbf{K}_6}.
\end{align*}
By using $|n|\leq|\nt|+|n^\sha|\lesssim\varepsilon+2\nb^-\leq C$, we have
\begin{align*}
	\abs{\textbf{K}_1}\lesssim \norm{\ut_x}^2.
\end{align*}
It follows from \eqref{nsha-bound}, \eqref{usha-bound}, \eqref{n-usha-1st-bound} and the Young inequality that
\begin{align*}
	\abs{\textbf{K}_2}+\abs{\textbf{K}_4}\lesssim(\delta+\nu e^{-\alpha t})\norm{[\nt,\ut]}_{H^1}^2,
\end{align*}
and 
\begin{align*}
	\abs{\textbf{K}_3}\leq\left|\left(\phit_x,\nt_x\right)\right|+\left|\left(\phi^\sha_x\nt,\frac{\nt_x}{n}\right)\right|\lesssim\eta_2\norm{\nt_x}+\frac{1}{\eta_2}\norm{\phit_x}^2+(\delta+\nu e^{-\alpha t})\norm{\nt}_{H^1}^2.
\end{align*}
By using Sobolev embedding and \eqref{eqn-assum-shock}, we have
\begin{align*}
	\abs{\textbf{K}_5}\lesssim\norm{\ut_x}_{L^\infty}\norm{\nt_x}^2\lesssim\varepsilon\norm{\ut_x}_{H^1}^2+\varepsilon\norm{\nt_x}^2.
\end{align*}
By using estimates of $h_1,h_2$ and \eqref{eqn-assum-shock}, we have
\begin{align*}
	\abs{\textbf{K}_6}\lesssim\norm{[h_1,\p_xh_1,h_2]}\norm{[\nt_x,\nu_x]}\lesssim\varepsilon\delta^\frac{1}{2}\nu e^{-\alpha t}.
\end{align*}
Thus,
\begin{align}\label{n-1st-shock-energy}
	&\frac{1}{2}\frac{d}{dt}\left(\frac{\nt_x}{n^2},\nt_x\right)+\frac{d}{dt}\left(\ut,\nt_x\right)+\frac{A}{\nb^+}\norm{\nt_x}\nonumber\\
	&\lesssim\norm{\phit_x}^2+\norm{\ut_x}^2+(\delta+\nu e^{-\alpha t})\norm{[\nt,\ut]}_{H^1}^2+\varepsilon\delta^\frac{1}{2}\nu e^{-\alpha t}.
\end{align}
Taking suitable linear combination of \eqref{u-0th-shock-energy},\eqref{u-1st-shock-energy} and \eqref{n-1st-shock-energy}, integrating with respect to time in $[0,\tau]$, choosing $\delta_0,\nu_0,\varepsilon_0$ small and using \eqref{ineq-1st-shock}, \eqref{1st-shock-energy-phit}, we can obtain that
\begin{align}\label{2nd-shock-energy}
	&\sup_{t\in[0,T]}\norm{[\ut,\ut_x,\nt_x]}+\int_{0}^{T}\norm{[\nt_x,\ut_x,\p_x^2\ut]}^2dt\nonumber\\
	&\lesssim \norm{[\ut_0,\ut_{0x},\nt_{0x}]} +\int_{0}^{T}\norm{\phit_x}^2 dt+(\delta+\nu)\int_{0}^{T}\norm{[\nt,\ut]}^2 dt+\nu\\
	&\lesssim\norm{[\Phi_0,\Psi_0]}_{H^2}^2
	 +(\delta+\varepsilon+\nu) \int_{0}^{T} \norm{\phit_t}^2 \, dt +\nu .\nonumber
\end{align}
Finally, we take the derivative $\p_t$ on \subeqref{eqn-pert-shock-new}{3} and inner product with $-\p_t\phit$ to get
\begin{align*}
	(\p_t\phit_x,\p_t\phit_x)+(e^{-\phisinf}\p_t\phit,\p_t\phit)=-\underbrace{(e^{-\phisinf}\p_t\phisinf\phit,\p_t\phit)}_{\textbf{K}_7}-\underbrace{(\p_t\nt,\p_t\phit)}_{\textbf{K}_8}-\underbrace{(\p_th_3+\p_tq_5+\p_tq_6,\p_t\phit)}_{\textbf{K}_9}.
\end{align*}
It follows from \eqref{eqn-shock-deriv} and the Young inequality that
\begin{align*}
	\abs{{\textbf{K}_7}}&\lesssim\eta\norm{\p_t\phit}^2+\frac{\delta}{\eta}\norm{\phit}^2,\\
	\abs{{\textbf{K}_8}}&\lesssim\eta\norm{\p_t\phit}^2+\frac{1}{\eta}\norm{\p_t\nt}^2\lesssim\eta\norm{\p_t\phit}^2+\frac{1}{\eta}(\norm{h_1}^2+\norm{\mt}^2)\\
	&\lesssim\eta\norm{\p_t\phit}^2+\frac{1}{\eta}(\nu e^{-\alpha t}+\norm{[\nt,\ut]}_{H^1}^2),\\
	\abs{{\textbf{K}_9}}&\lesssim\nu e^{-\alpha t}\norm{\p_t\phit}^2+\norm{\phit}_{L^\infty}\norm{\p_t\phit}^2+\nu e^{-\alpha t}\lesssim(\nu+\varepsilon)\norm{\p_t\phit}^2+\nu e^{-\alpha t}.
\end{align*}
By choosing $\eta$ small enough, we obtain 
\begin{align*}
	\int_{0}^{T}\norm{\p_t\phit}^2dt\lesssim \int_{0}^{T}(\norm{[\nt,\ut]}_{H^1}^2+\norm{\phit}^2) dt+\nu.
\end{align*}
Plugging this into \eqref{2nd-shock-energy} yields that
\begin{align*}
	&\sup_{t\in[0,T]}\norm{[\ut,\ut_x,\nt_x]}+\int_{0}^{T}\norm{[\nt_x,\ut_x,\p_x^2\ut]}^2\,dt \lesssim \norm{[\Phi_0,\Psi_0]}_{H^2}^2+\nu .
\end{align*}
Therefore, \eqref{ineq-2nd-shock} follows.
\end{proof}

Now we will complete the proof the Theorem \ref{thm-shock}.
	\begin{proof}[Proof of Proposition \ref{prop-apriori-shock} and Theorem \ref{thm-shock}]
		Proposition \ref{prop-apriori-shock} is proved by combining the estimates \eqref{ineq-0th-shock}, \eqref{ineq-1st-shock}, \eqref{ineq-2nd-shock}, and \eqref{ineq-2nd-shock-phit}. Hence, Theorem \ref{thm-shock} follows from the standard local-existence theory, Proposition \ref{prop-apriori-shock} and the fact of the ansatz
		\begin{equation}
			\norm{[n^\sha,m^\sha,\phi^\sha](\cdot,t) - [n^s,m^s,\phi^s](\cdot-st-\X_\infty) }_{L^\infty(\R)} \leq \nu e^{-\alpha t}.
		\end{equation}
	\end{proof}

\section{Stability of rarefaction wave}
The proof of Theorem \ref{thm-rare} is based on the following \textit{a priori} estimates for the solution $[\nt, \ut, \phit]$ to \eqref{eqn-pert-rare}-\eqref{init-rare2}, since the local existence of $[\nt, \ut, \phit]$ follows from the standard iteration method, which is omitted here for simplicity.

\begin{proposition}[\textit{A priori} estimates]
	\label{prop-apriori-rare}
	For any $T>0$, assume that $[\nt, \ut, \phit]$ is a smooth solution to \eqref{eqn-pert-rare} and \eqref{init-rare2}. Then there exist $\nu_0>0$, $\varepsilon_0>0$, $0<\theta<1$ independent of $T$ such that if $\nu<\nu_0$, and 
	\begin{equation}
		\label{eqn-assum-rare}
		\varepsilon:=\sup\limits_{t\in[0,T]} \norm{[\nt, \ut, \phit]}_{H^1}+\epsilon < \varepsilon_0,
	\end{equation}
	then we have 
	\begin{equation}
		\label{eqn-aprioi-rare}
		\begin{aligned}
			\sup\limits_{t\in[0,T]} \norm{[\nt, \ut, \phit]}_{H^1}^2+
			&\int_{0}^{T} \Big(\norm{ \abs{ u^r_x}^{1/2}\ut}^2 + \norm{[\nt_x, \ut_x, \p_x^2\ut, \phit_x,\p_x^2\phit]}^2 \Big) \, dt\\
			&\qquad\qquad
			\lesssim  \norm{[\nt_0,\ut_0]}_{H^1}^2 +\epsilon^\frac{\theta}{1+\theta}+\nu .
		\end{aligned}
	\end{equation}
\end{proposition} 
Proposition \ref{prop-apriori-rare} is a consequence of the following lemmas of energy estimates.
We start from zero-order energy estimate which is stated in the following lemma. 
\begin{lemma}
	\label{lem-0th-rare}
	Under the assumptions of Proposition \ref{prop-apriori-rare}, there exist positive constants $\nu_0>0$, $\varepsilon_0>0$, $0<\theta<1$ independent of $T$ such that if $\nu<\nu_0$, and \eqref{eqn-assum-rare} holds, then
    	\begin{equation}
    	\label{ineq-0th-rare}
    	\begin{aligned}
    		&\sup\limits_{t\in[0,T]} \norm{[\ut, \nt,\phit,\phit_x]}^2+
    		\int_{0}^{T} \norm{ [\abs{u^r_x}^{\frac{1}{2}}\ut, \ut_x] }^2 \, dt  \\
    		&
    		\lesssim  \norm{[\nt_0,\ut_0]}^2 + \norm{ \phit_0}_{H^1}^2  +(\epsilon+\nu+\eta)\int_{0}^{T}\norm{\phit_x}^2\,dt+\epsilon^\frac{\theta}{1+\theta} +\nu,
    	\end{aligned}
    \end{equation}
	where $\eta>0$ is a constant to be determined later.
\end{lemma}
\begin{proof}
	Set $$\Lambda(n,n^\sha)=A\int^{n}_{n^\sha}\frac{s-n^\sha}{s^2}ds=A\left(\ln\left(\frac{n}{n^\sha}\right)+\frac{n^\sha}{n}-1\right).$$
	Taking the $L^2$-inner products between  \subeqref{eqn-pert-rare}{2} and $n\ut$ as well as \subeqref{NSP-Ion}{1} and $\Lambda(n,n^\sha)$, and summing the resulting equation together, we have
	\begin{align}\label{0th-rare-energy}
		&\frac{d}{dt}\left\{\frac{1}{2}\left(n\ut,\ut\right)+\left(n,\Lambda\right)\right\}+\left(\ut^2,nu^r_x\right)-\left(\phit_x,n\ut\right)+(\ut_x,\ut_x)\nonumber\\
		&=-\left(\ut^2,n(u^\sha-u^r)_x\right)-\left(\frac{\p_x^2(u^\sha-u^r)}{n^\sha}\nt,\ut\right)-\left(\frac{\p_x^2u^r}{n^\sha}\nt,\ut\right)+\left(\frac{\p_x^2 u^r}{n^r},n\ut\right)\\
		&\quad-A\left(\frac{\nt}{n^\sha},k_1\right)+(k_2,n\ut),\nonumber
	\end{align} 
where we have used the following identities
\begin{align*}
	-\frac{1}{2}\left(\ut^2,n_t\right)+(u\p_xu-u^\sha\p_xu^\sha,n\ut)=\left(\ut^2,nu^\sha_x\right)=\left(\ut^2,nu^r_x\right)+\left(\ut^2,n(u^\sha-u^r)_x\right),
\end{align*}
\begin{align*}
	(n,-\Lambda_t)+((nu)_x,\Lambda)+A\left(\frac{n_x}{n}-\frac{n^\sha_x}{n^\sha},n\ut\right)=A\left(n,\left(\frac{1}{n}-\frac{1}{n^\sha}\right)k_1\right)=-A\left(\frac{\nt}{n^\sha},k_1\right),
\end{align*}
and
\begin{align*}
	\left(\frac{\p_x^2u}{n}-\frac{\p_x^2u^\sha}{n^\sha},n\ut\right)&=\left(\p_x^2\ut,\ut\right)-\left(\frac{\p_x^2u^\sha}{n^\sha}\nt,\ut\right)\\
	&=-(\p_x\ut,\p_x\ut)-\left(\frac{\p_x^2(u^\sha-u^r)}{n^\sha}\nt,\ut\right)-\left(\frac{\p_x^2u^r}{n^\sha}\nt,\ut\right),
\end{align*}
which follows from integration by parts and \subeqref{eqn-pert-rare}{1}. For the third term on the left hand side of \eqref{0th-rare-energy}, we utilize integration by parts, \subeqref{eqn-pert-rare}{1} and \subeqref{eqn-pert-rare}{3} to obtain that
\begin{align*}
	&-\left(\phit_x,n\ut\right)=\left(\phit,(n\ut)_x\right)=\left(\phit,-\nt_t+k_1-(\nt u^\sha)_x\right)\\
	&=\left(\phit,-\p_t\p_x^2\phit\right)+\left(\phit,\p_t\left(e^{-\phi^\sha}\left(1-e^{-\phit}\right)\right)\right)-(\phit,\p_t\p_x^2\phi^r)+\left(\phit,\p_tk_3+k_1-\left(\nt u^\sha\right)_x\right)\\
	&=\frac{1}{2}\frac{d}{dt}\left(\phit_x,\phit_x\right)+\left(\phit,\p_t\left(e^{-\phi^\sha}\left(1-e^{-\phit}\right)\right)\right)-\left(\phit,\left(\nt u^\sha\right)_x\right)-\left(\phit,\p_t\p_x^2\phi^r\right)+\left(\phit,\p_tk_3+k_1\right),
\end{align*} 
and
\begin{align*}
	&\left(\phit,\p_t\left(e^{-\phi^\sha}\left(1-e^{-\phit}\right)\right)\right)-\left(\phit,\left(\nt u^\sha\right)_x\right)=\left(\phit,\p_t\left(e^{-\phi^\sha}\left(1-e^{-\phit}\right)\right)\right)+\left(\phit_x,\nt u^\sha\right)\\
	&=\left(\phit,\p_t\left(e^{-\phi^\sha}\left(1-e^{-\phit}\right)\right)\right)-\left(\phit_x\left(1-e^{-\phit}\right),e^{-\phi^\sha}u^\sha\right)+\left(\phit_x,\p_x^2\phit u^\sha\right)+\left(\phit_x,\left(\p_x^2\phi^r-k_3\right)u^\sha\right)\\
	&=\frac{d}{dt}\left(\phit,e^{-\phi^\sha}\left(1-e^{-\phit}\right)\right)-\left(e^{-\phi^\sha},\p_t\left(\phit+e^{-\phit}\right)\right)-\left(e^{-\phi^\sha}u^\sha,\p_x\left(\phit+e^{-\phit}\right)\right)\\
	&\quad-\frac{1}{2}\left(u^\sha_x,\phit_x^2\right)+\left(\phit_x,\left(\p_x^2\phi^r-k_3\right)u^\sha\right)\\
	&=\frac{d}{dt}\left(e^{-\phi^\sha},1-e^{-\phit}-\phit e^{-\phit}\right)+\left(\p_te^{-\phi^\sha}+\p_x\left(e^{-\phi^\sha}u^\sha\right),\phit+e^{-\phit}-1\right)\\
	&\quad-\frac{1}{2}\left(u^\sha_x,\phit_x^2\right)+\left(\phit_x,\left(\p_x^2\phi^r-k_3\right)u^\sha\right).
\end{align*}  
Thus,
\begin{align*}
	-\left(\phit_x,n\ut\right)=&\frac{1}{2}\frac{d}{dt}\left(\phit_x,\phit_x\right)+\frac{d}{dt}\left(e^{-\phi^\sha},1-e^{-\phit}-\phit e^{-\phit}\right)\\
	&-\left(\phit,\p_t\p_x^2\phi^r\right)+\left(\p_te^{-\phi^\sha}+\p_x\left(e^{-\phi^\sha}u^\sha\right),\phit+e^{-\phit}-1\right)\\
	&-\frac{1}{2}\left(u^\sha_x,\phit_x^2\right)+\left(\phit,\p_tk_3+k_1\right)+(\p_x\phit,(\p_x^2\phi^r+k_3)u^\sha).
\end{align*}
Therefore, we obtain that
\begin{equation}\label{0th-rare-1}
	\begin{aligned}
	&\frac{d}{dt}\left\{\frac{1}{2}\left(n\ut,\ut\right)+\left(n,\Lambda\right)+\frac{1}{2}(\phit_x,\phit_x)+(e^{-\phi^\sha},1-e^{-\phit}-\phit e^{-\phit})\right\}+\left(\ut^2,nu^r_x\right)+(\ut_x,\ut_x)\\
	&=\underbrace{-\left(\ut^2,n(u^\sha-u^r)_x\right)-\left(\frac{\p_x^2(u^\sha-u^r)}{n^\sha}\nt,\ut\right)}_{\textbf{I}_1}+\underbrace{\left(\frac{\p_x^2 u^r}{n^r},n\ut\right)-\left(\frac{\p_x^2u^r}{n^\sha}\nt,\ut\right)}_{\textbf{I}_2}\\
	&\quad+\underbrace{(\phit,\p_t\p_x^2\phi^r)}_{\textbf{I}_3}+\underbrace{\frac{1}{2}\left(u^\sha_x,\phit_x^2\right)}_{\textbf{I}_4}-\underbrace{(\p_x\phit,\p_x^2\phi^ru^\sha)}_{\textbf{I}_5}-\underbrace{(\p_te^{-\phi^\sha}+\p_x(e^{-\phi^\sha}u^\sha),\phit+e^{-\phit}-1)}_{\textbf{I}_6}\\
	&\quad\underbrace{-(\p_x\phit,k_3u^\sha)-A\left(\frac{\nt}{n^\sha},k_1\right)+(k_2,n\ut)-(\phit,\p_tk_3+k_1)}_{\textbf{I}_7}.
\end{aligned} 
\end{equation} 
To estimate $\textbf{I}_i$, we first give estimates of $[n^\sha,u^\sha,\phi^\sha]$. It follows from \eqref{def-ansatz-rare} and Lemma \ref{decay-per} that
\begin{equation}\label{n-u-phi-bdd}
	\norm{[n^\sha-n^r,u^\sha-u^r,\phi^\sha-\phi^r]}_{W^{2,\infty}}\lesssim\norm{[\rho^\pm,w^\pm,\varphi^\pm]}_{W^{2,\infty}}\lesssim\nu e^{-\alpha t}.
\end{equation}
If $\nu,\varepsilon$ sufficiently small, then
\begin{equation}\label{n-bdd}
	\frac{\nb^-}{2}\leq n^\sha\leq \frac{3}{2}\nb^+,\quad \frac{\nb^-}{4}\leq n\leq 2\nb^+ ,
\end{equation}
and
\begin{equation}\label{u-bdd}
	|u^\sha|\leq |u^r|+\nu e^{-\alpha t}\leq\max\{ |u^+|,|u^-|\}+\nu\leq C.
\end{equation}
Then, it follows from \eqref{n-u-phi-bdd}, \eqref{n-bdd} and Lemma \ref{lem-rare-prop} that
\begin{align*}
	\abs{\textbf{I}_1}&\lesssim\norm{(u^\sha-u^r)_x}_{L^\infty}\norm{\ut}^2+\norm{\p_x^2(u^\sha-u^r)}_{L^\infty}\norm{\nt}\norm{\ut}\\
	&\lesssim\nu e^{-\alpha t}\norm{[\nt,\ut]}^2.
\end{align*}
It follows from \eqref{n-bdd}, the H\"{o}lder inequality, the Young inequality, Lemma \ref{lem-rare-prop} and the Sobolev inequality that, for a constant $0<\theta<1$,
\begin{align*}
	\abs{\textbf{I}_2}&\lesssim \norm{\p_x^2u^r}_{L^{1+\theta}}\|\ut\|_{L^{\frac{1+\theta}{\theta}}}\lesssim\norm{\p_x^2u^r}_{L^{1+\theta}}\|\ut\|_{L^\infty}^\frac{1-\theta}{1+\theta}\norm{\ut}^\frac{2\theta}{1+\theta}\\
	&\lesssim\epsilon^\frac{\theta}{1+\theta}(1+t)^{-1}\norm{\ut_x}^\frac{1-\theta}{2(1+\theta)}\norm{\ut}^{\frac{1+3\theta}{2(1+\theta)}}\\
	&\lesssim\epsilon^\frac{\theta}{1+\theta}\left(\norm{\ut_x}^2+(1+t)^{-\frac{5+3\theta}{4+4\theta}}+(1+t)^{-\frac{3+5\theta}{2(1+3\theta)}}\norm{\ut}^2\right).
\end{align*}
We can obtain from integration by parts, Lemma \ref{lem-rare-prop} and the Young inequality that
\begin{align*}
	\abs{\textbf{I}_3}=\left|(\p_x\phit,\p_t\p_x\phi^r)\right|\lesssim\eta\norm{\phit_x}^2+C_\eta\norm{\p_t\p_x\phi^r}^2\lesssim\eta\norm{\phit_x}^2+\epsilon C_\eta(1+t)^{-2}.
\end{align*} 
It follows from \eqref{n-u-phi-bdd} and Lemma \ref{lem-rare-prop}
\begin{align*}
	\abs{\textbf{I}_4}&\lesssim\left(\norm{u^r_x}_{L^\infty}+\norm{(u^\sha-u^r)_x}_{L^\infty}\right)\norm{\p_x\phit}^2\lesssim(\delta_r\epsilon+\nu e^{-\alpha t})\norm{\p_x\phit}^2,
\end{align*}
\begin{align*}
	\abs{\textbf{I}_5}\lesssim (1+\nu e^{-\alpha t})\norm{\p_x^2\phi^r}\norm{\phit_x}\lesssim \eta\norm{\phit_x}^2+\epsilon C_\eta(1+t)^{-2}.
\end{align*}
By using \eqref{def-ansatz-rare}, $\phi^r=-\ln n^r$ and Lemma \ref{lem-rare-prop}, it is easy to obtain that 
\begin{align*}
	\norm{\p_t\left(e^{-\phi^\sha}-n^r\right)}_{L^\infty}+\norm{e^{-\phi^\sha}-n^r}_{W^{1,\infty}}\lesssim\nu e^{-\alpha t}.
\end{align*}
Thus, we can obtain that
\begin{align*}
	\abs{\textbf{I}_6}\lesssim\abs{(\p_t(e^{-\phi^\sha}-n^r)+\p_x((e^{-\phi^\sha}-n^r)u^\sha+n^r(u^\sha-u^r)),\frac{1}{2}\phit^2+\phit^3)}\lesssim \nu e^{-\alpha t}\norm{\phit}^2.
\end{align*}
For the problem of the stability of rarefaction waves, $k_i$ $(i=1,2,3)$ are in the form of $\sR=\sR_{\rm Ti} \sR_{\rm Sp}$ (see \eqref{eqn-error-rare}), and $\sR_{\rm Ti}$ are the periodic functions and their derivatives, and $\sR_{\rm Sp}$ are the derivatives of the background rarefaction waves. Therefore, it follows from Lemma \ref{Lem-periodic} and \ref{lem-rare-prop} that 
\begin{equation}\label{ki-est}
	\norm{k_i}_{H^2} \lesssim (\norm{\p_x^2 \sR_{\rm Ti} }_{L^\infty} \norm{\sR_{\rm Sp}}+\cdots + \norm{ \sR_{\rm Ti} }_{L^\infty} \norm{\p_x^2\sR_{\rm Sp}} )\lesssim\nu e^{-\alpha t}.
\end{equation}
Then, by using \eqref{n-bdd}, \eqref{u-bdd},  and the Young inequality, we have
\begin{align*}
	\abs{\textbf{I}_{7}}\lesssim\nu e^{-\alpha t}\norm{[\nt,\ut,\phit,\phit_x]}^2+\nu e^{-\alpha t}.
\end{align*}
Note that $\Lambda(n,n^\sha)\sim|\nt|^2$ if $|\nt|$ is bounded,  $\p_xu^r>0$ and 
\begin{align*}
	(e^{-\phi^\sha},1-e^{-\phit}-\phit e^{-\phit})&=(n^r,1-e^{-\phit}-\phit e^{-\phit})+(e^{\phi^\sha}-n^r,1-e^{-\phit}-\phit e^{-\phit})\\
	&\gtrsim(n^r-\nu e^{-\alpha t})\|\phit\|^2
	\gtrsim\frac{\nb^-}{2}\|\phit\|^2,
\end{align*}
for any $\|\phit\|_{L^\infty}\leq C$ and $\nu$ sufficiently small. Then, integrating \eqref{0th-rare-1} in time over $[0,t]$, and plugging into estimates of $I_i (i=1,\cdots,7)$ give that for small enough $\eta$,
\begin{equation}
	\label{ineq-0th-rare'}
	\begin{aligned}
		&\sup\limits_{t\in[0,T]} \norm{[\ut, \nt,\phit,\phit_x]}^2+
		\int_{0}^{T} \norm{ [\sqrt{u^r_x}\ut, \ut_x] }^2 \, dt  \\
		&
		\lesssim  \norm{[\nt_0,\ut_0]}^2 + \norm{ \phit_0}_{H^1}^2 +  	\int_{0}^{T}\nu e^{-\alpha t}\norm{[\nt,\ut,\phit]}^2 \, dt+  \int_{0}^{T}\epsilon^\frac{\theta}{1+\theta}(1+t)^{-\frac{3+5\theta}{2(1+3\theta)}}\norm{\ut}^2\,dt\\
		&\qquad+\epsilon^\frac{\theta}{1+\theta}	\int_{0}^{T}  \norm{\ut_x}^2 \, dt+(\epsilon+\nu+\eta)\int_{0}^{T}\norm{\phit_x}^2\,dt+\epsilon^\frac{\theta}{1+\theta} +\nu.
	\end{aligned}
\end{equation}
Therefore, \eqref{ineq-0th-rare} follows with small $\eta$, $\nu$ and $\epsilon$.
\end{proof}

\begin{lemma}
	\label{lem-1st-rare}
	Under the assumptions of Proposition \ref{prop-apriori-rare}, there exist positive constants $\nu_0>0$, $\varepsilon_0>0$, $0<\theta<1$ independent of $T$ such that if $\nu<\nu_0$, and \eqref{eqn-assum-rare} holds, then
	\begin{equation}
		\label{ineq-1st-rare}
		\begin{aligned}
			&\sup\limits_{t\in[0,T]} \norm{\nt_x}^2+
			\int_{0}^{T} \norm{ [\nt_x, \phit_x, \p_x^2\phit] }^2 \, dt  \\
			&\qquad\qquad
			\lesssim  \norm{[\ut_0]}^2 + \norm{ [\nt_0,\phit_0]}_{H^1}^2  
			+  \varepsilon \int_{0}^{T} \norm{  \p_x^2\ut}^2 \, dt 
			+ \epsilon^\frac{\theta}{1+\theta} + \nu .
		\end{aligned}
	\end{equation}
\end{lemma}
\begin{proof}
	Differentiating \subeqref{eqn-pert-rare}{1} and \subeqref{eqn-pert-rare}{3} in $x$, taking the inner products of the resultant  with $\nt_x/n^2$ and $\phit_x$ with respect to $x$ respectively, and using integration by parts yield that
\begin{align}
	\label{1st-rare-1}
		&\frac{1}{2}\frac{d}{dt}(\nt_x^2,n^{-2})+\frac{1}{2}\left(u^\sha_x,\frac{\nt_x^2}{n^2}\right)+(\p_x^2\phit,\p_x^2\phit)+(e^{-\phi^\sha}e^{-\phit}\phit_x,\phit_x)+\left(\p_x^2\ut,\frac{\nt_x}{n}\right)+(\nt_x,\phit_x)\nonumber\\
		&=-\frac{1}{2}\left(\frac{\ut_x}{n^2},\nt_x^2\right)-\left(\p_x^2u^\sha\nt+\p_x^2n^\sha\ut+2n^\sha_x\ut_x,\frac{\nt_x}{n^2}\right)+\left(\p_xk_1,\frac{\nt_x}{n^2}\right)\\
		&\quad-\left(\p_x(e^{-\phi^\sha})(1-e^{-\phit}),\p_x\phit\right)+(\p_x^3\phi^r,\phit_x)-(\p_xk_3,\phit_x),\nonumber
\end{align} 
where we have used 
\begin{align*}
	-(\nt_x^2,n^{-3}\p_tn)-(u_x\nt_x+u\p_x^2\nt+n_x\ut_x,\frac{\nt_x}{n^2})=\frac{1}{2}\left(u^\sha_x,\frac{\nt_x}{n^2}\right)-\frac{1}{2}\left(\frac{\ut_x}{n^2},\nt_x^2\right)-(n^\sha_x\ut_x,\frac{\nt_x}{n^2}) .
\end{align*}
To cancel the last two terms on the left hand side of \eqref{1st-rare-1}, we take the inner product of \subeqref{eqn-pert-rare}{2} with $\nt_x$ and sum with \eqref{1st-rare-1} to obtain that
\begin{align}
	\label{1st-rare-2}
	&\frac{d}{dt}(\ut,\nt_x)+\frac{1}{2}\frac{d}{dt}(\nt_x^2,n^{-2})+\frac{1}{2}\left(u^\sha_x,\frac{\nt_x^2}{n^2}\right)+(\p_x^2\phit,\p_x^2\phit)+(e^{-\phi^\sha}e^{-\phit}\phit_x,\phit_x)+A(n^{-1},\nt_x^2)\nonumber\\
	&=-\underbrace{\frac{1}{2}\left(\frac{\ut_x}{n^2},\nt_x^2\right)}_{\textbf{J}_1}-\underbrace{\left(\p_x^2u^\sha\nt+\p_x^2n^\sha\ut+2n^\sha_x\ut_x,\frac{\nt_x}{n^2}\right)}_{\textbf{J}_2}+\underbrace{(\ut,\p_t\nt_x)-(uu_x-u^\sha u^\sha_x,\nt_x)}_{\textbf{J}_3}\nonumber\\
	&\quad+\underbrace{A\left(\frac{\nt n^\sha_x}{nn^\sha},\nt_x\right)-\left(\frac{\nt\p_x^2u^\sha}{nn^\sha},\nt_x\right)+\left(\frac{\p_x^2u^r}{n^r},\nt_x\right)}_{\textbf{J}_4}-\underbrace{\left(\p_x(e^{-\phi^\sha})(1-e^{-\phit}),\p_x\phit\right)}_{\textbf{J}_5}\\
	&\quad+\underbrace{(\p_x^3\phi^r,\phit_x)}_{\textbf{J}_6}+\underbrace{\left(\p_xk_1,\frac{\nt_x}{n^2}\right)+(k_2,\nt_x)-(\p_xk_3,\phit_x)}_{\textbf{J}_7}.\nonumber
\end{align}
The terms $\textbf{J}_1,\cdots,\textbf{J}_7$ can be bounded as follows. It follows from the Young inequality, Sobolev embedding and \eqref{eqn-assum-rare} that 
\begin{align*}
	\abs{\textbf{J}_1}&\lesssim\eta\norm{\nt_x}^2+C_\eta\norm{\nt_x}^2\norm{\ut_x}_{L^\infty}^2\lesssim(\eta+C_\eta\varepsilon^2)\norm{\nt_x}^2+C_\eta\varepsilon^2\norm{\p_x^2\ut}^2.
\end{align*}
By using \eqref{n-u-phi-bdd} and  Lemma \ref{lem-rare-prop}, one has
\begin{align*}
	\abs{\textbf{J}_2}&\lesssim\eta\norm{\nt_x}^2+C_\eta\norm{n^\sha_x}_{L^\infty}^2\norm{\ut_x}^2+C_\eta(\norm{\p_x^2u^\sha}_{L^\infty}^2\norm{\nt}^2+\norm{\p_x^2n^\sha}_{L^\infty}^2\norm{\ut}^2)\\
	&\lesssim\eta\norm{\nt_x}^2+C_\eta(\delta^2\epsilon+\nu e^{-\alpha t})\norm{\ut_x}^2+C_\eta(\epsilon^2+\nu)(1+t)^{-2}\|[\nt,\ut]\|^2.
\end{align*}
To estimate $\textbf{J}_3$, we utilize integration by parts and \subeqref{eqn-pert-rare}{1} to derive that
\begin{align*}
\textbf{J}_3&=(\ut_x,\p_x(nu-n^\sha u^\sha))-(\ut_x,k_1)-(uu_x-u^\sha u^\sha_x,\nt_x)\\
	&=(n,\ut_x^2)+(\ut_x,\nt u^\sha_x)+(\ut_x,n^\sha_x\ut)-(\ut,u^\sha_x\nt_x)-(\ut_x,k_1),
\end{align*}
so that we can obtain from \eqref{n-bdd} that
\begin{align*}
	\abs{\textbf{J}_3}&\lesssim\|n\|_{L^\infty}\|\ut_x\|^2+\eta\|[\nt_x,\ut_x]\|^2+C_\eta(\|u^\sha_x\|_{\infty}^2\|\nt\|^2+\|[n^\sha_x,u^\sha_x]\|_{L^\infty}^2\|\ut\|^2+\|k_1\|^2)\\
	&\lesssim\eta\|\nt_x\|^2+(\eta+\bar{n}^+)\norm{\ut_x}^2+C_\eta(1+t)^{-2}\|[\nt,\ut]\|^2+C_\eta\nu e^{-\alpha t}.
\end{align*}
In view of \eqref{n-u-phi-bdd} and Lemma \ref{lem-rare-prop}, it is easy to obtain that
\begin{align*}
		\abs{\textbf{J}_4}&\lesssim\eta\|\nt_x\|^2+C_\eta\|\nt\|^2(\|n^\sha_x\|_{L^\infty}^2+\|\p_x^2u^\sha\|_{L^\infty}^2)+C_\eta\|\p_x^2u^r\|^2\\
	&\lesssim\eta\|\nt_x\|^2+C_\eta(1+t)^{-2}\|\nt\|^2+C_\eta\epsilon(1+t)^{-2},
\end{align*}
\begin{align*}
	\abs{\textbf{J}_5}&\lesssim\eta\|\phit_x\|^2+C_\eta(\|\p_xn^r\|_{L^\infty}^2+\|\p_x(e^{-\phi^\sha}-n^r)\|_{L^\infty}^2)\|\phit\|^2\lesssim\eta\|\phit_x\|^2+C_\eta(1+t)^{-2}\|\phit\|^2,
\end{align*}
and
\begin{align*}
	\abs{\textbf{J}_6}&\lesssim\eta\|\phit_x\|^2+C_\eta\|\p_x^3\phi^r\|^2\lesssim\eta\|\phit_x\|^2+C_\eta\epsilon^3(1+t)^{-2}.
\end{align*} 
By using \eqref{ki-est}, we have
\begin{align*}
	\abs{\textbf{J}_7}&\lesssim\eta\|[\nt_x,\phit_x]\|^2+C_\eta\|[\p_xk_1,k_2,\p_xk_3]\|^2\lesssim\eta\|[\nt_x,\phit_x]\|^2+C_\eta\nu e^{-\alpha t}.
\end{align*}
Note that
\begin{align*}
	\left|\int_{0}^{T}\frac{d}{dt}(\ut,\nt_x)\right|\lesssim \|\ut_0\|^2+\|\nt_{0x}\|^2+\eta\|\nt_x(T)\|^2+C_\eta\|\ut(T)\|^2,
\end{align*}
and $\p_x u^r >0$, we have
\begin{align*}
	\left(u^\sha_x,\frac{\nt_x^2}{n^2}\right)=\left(u^r_x,\frac{\nt_x^2}{n^2}\right)+\left(u^\sha_x-u^r_x,\frac{\nt_x^2}{n^2}\right) \gtrsim -\nu e^{-\alpha t} \|\nt_x\|^2,
\end{align*}
\begin{align*}
	(e^{-\phi^\sha}e^{-\phit}\phit_x,\phit_x)=(n^re^{-\phit}\phit_x,\phit_x)+((e^{-\phi^\sha}-n^r)e^{-\phit}\phit_x,\phit_x)\gtrsim (1-\nu e^{-\alpha t})\|\phit_x\|^2.
\end{align*} 
Then, by integrating \eqref{1st-rare-2} with respect to time on $[0,t]$ and plugging estimates of $\textbf{J}_i (i=1,\cdots,7)$, we have that for small enough $\eta$,
\begin{equation}
	\label{ineq-1st-rare'}
	\begin{aligned}
		&\sup\limits_{t\in[0,T]} \norm{\nt_x}^2+
		\int_{0}^{T} \norm{ [\nt_x, \phit_x, \p_x^2\phit] }^2 \, dt  \\
		&\qquad
		\lesssim  \norm{[\ut_0]}^2 + \norm{ [\nt_0,\phit_0]}_{H^1}^2   
		+\sup\limits_{t\in[0,T]} \norm{\ut}^2
		+ \int_{0}^{T} (1+t)^{-2} \norm{ [\nt,\ut, \phit] }^2 \, dt   \\
		&\qquad\quad
		+  \int_{0}^{T} \norm{ \p_x\ut }^2 \, dt
		+ \nu \int_{0}^{T} \norm{ [\p_x \nt, \p_x \phit] }^2 \, dt 
		+  \varepsilon\int_{0}^{T} \norm{ [\p_x \nt, \p_x^2\ut]}^2 \, dt 
		+ \epsilon + \nu .
	\end{aligned}
\end{equation}
Therefore, \eqref{ineq-1st-rare} follows by using \eqref{ineq-0th-rare} with small $\eta$, $\nu$ and $\varepsilon$.
\end{proof}
\begin{lemma}
	\label{lem-2nd-rare}
	Under the assumptions of Proposition \ref{prop-apriori-rare}, there exist positive constants $\nu_0>0$, $\varepsilon_0>0$, $0<\theta<1$ independent of $T$ such that if $\nu<\nu_0$, and \eqref{eqn-assum-rare} holds, then
	\begin{equation}
		\label{ineq-2nd-rare}
		\begin{aligned}
			&\sup\limits_{t\in[0,T]} \norm{\ut_x}^2+
			\int_{0}^{T} \norm{ \p_x^2\ut }^2 \, dt  
			\lesssim  \norm{[\nt_0, \ut_0, \phit_0]}_{H^1}^2
			+ \epsilon^\frac{\theta}{1+\theta}  + \nu .
		\end{aligned}
	\end{equation}
\end{lemma}
\begin{proof}
	Differentiating \subeqref{eqn-pert-rare}{2} in $x$, taking inner products of the resultant with $\p_x\ut$ with respect to $x$, and using integration by parts yield that
\begin{equation}\label{2nd-rare-1}
		\begin{aligned}
		&\frac{1}{2}\frac{d}{d}(\ut_x,\ut_x)+\frac{3}{2}(u^\sha_x,(\ut_x)^2)+\left(\frac{\p_x^2\ut}{n},\p_x^2\ut\right)\\
		&=\underbrace{(\ut\p_x^2\ut,\ut_x)-(\ut\p_x^2u^\sha,\ut_x)}_{\textbf{K}_1}-\underbrace{A\left(\p_x\left(\frac{n_x}{n}-\frac{n^\sha_x}{n^\sha}\right),\ut_x\right)}_{\textbf{K}_2}\\
		&\quad+\underbrace{\left(\p_x^2\phit,\ut_x\right)}_{\textbf{K}_3}+\underbrace{\left(\frac{\p_x^2u^\sha \nt}{nn^\sha},\p_x^2\ut\right)}_{\textbf{K}_4}\underbrace{-\left(\frac{\p_x^2n^r}{n^r},\p_x^2\ut\right)-(k_2,\p_x^2\ut)}_{\textbf{K}_5} ,
	\end{aligned}
\end{equation}
where we have used 
\begin{align*}
	(\p_x(uu_x-u^\sha u^\sha_x),\ut_x)&=(\p_x(\ut\ut_x+u^\sha\ut_x+\ut u^\sha_x),\ut_x)\\
	&=\frac{3}{2}(u^\sha_x,(\ut_x)^2)-(\ut\p_x^2\ut,\ut_x)+(\ut\p_x^2u^\sha,\ut_x),
\end{align*}
and
\begin{align*}
	\left(\p_x\left(\frac{\p_x^2u}{n}-\frac{\p_x^2u^\sha}{n}\right),\p_x\ut\right)=-\left(\frac{\p_x^2u}{n}-\frac{\p_x^2u^\sha}{n},\p_x^2\ut\right)=-\left(\frac{\p_x^2\ut}{n},\p_x^2\ut\right)+\left(\frac{\p_x^2u^\sha \nt}{nn^\sha},\p_x^2\ut\right).
\end{align*}
It is obvious from \eqref{n-u-phi-bdd} and Lemma \ref{lem-rare-prop} to get
\begin{align*}
	\abs{\textbf{K}_1}&\lesssim\eta(\norm{\ut_x}^2+\norm{\p_x^2\ut}^2)+C_\eta\norm{\ut}_{\infty}^2(\norm{\ut_x}^2+\norm{\p_x^2u^\sha}^2)\\
	&\lesssim\eta\norm{\p_x^2\ut}^2+(\eta+C_\eta\norm{\ut}_{H^1}^2)\norm{\ut_x}^2+C_\eta\norm{\ut}_{H^1}^2(\epsilon(1+t)^{-2}+\nu e^{-\alpha t}),\\
	\abs{\textbf{K}_3}&\lesssim\norm{\ut_x}^2+\norm{\p_x^2\phit},\\
	\abs{\textbf{K}_4}&\lesssim\eta\norm{\p_x^2\ut}^2+C_\eta\|\nt\|_{L^\infty}^2\|\p_x^2u^\sha\|^2\lesssim\eta\norm{\p_x^2\ut}^2+C_\eta\norm{\nt}_{H^1}^2(\epsilon(1+t)^{-2}+\nu e^{-\alpha t}),\\
	\abs{\textbf{K}_5}&\lesssim\eta\norm{\p_x^2\ut}^2+C_\eta(\norm{\p_x^2n^r}^2+\norm{k_2}^2)\lesssim\eta\norm{\p_x^2\ut}^2+C_\eta
	(\epsilon(1+t)^{-2}+\nu e^{-\alpha t}).
\end{align*}
For $\textbf{K}_2$, it follows from integration by parts that
\begin{align*}
	\textbf{K}_2&=-A\left(\frac{\nt_x}{n}-\frac{n^\sha_x\nt}{nn^\sha},\p_x^2\ut\right)=-A\left(\frac{\p_x\nt}{n},\p_x^2\ut\right)-A\left(\p_x\left(\frac{\p_xn^\sha\nt}{nn^\sha}\right),\p_x\ut\right)\\
	&=-A\left(\frac{\nt_x}{n},\p_x^2\ut\right)-A\left(\frac{\p_x^2n^\sha\nt}{nn^\sha},\p_x\ut\right)-A\left(\frac{n^\sha_x\p_x\nt}{nn^\sha},\ut_x\right)\\
	&\quad+A\left(\frac{(n^\sha_x)^2\nt}{n(n^\sha)^2},\p_x\ut\right)+A\left(\frac{n^\sha_x\nt_x\nt}{n^2n^\sha},\ut_x\right)+A\left(\frac{(n^\sha_x)^2\nt}{n^2n^\sha},\ut_x\right).
\end{align*}
Then, 
\begin{align*}
	\abs{\textbf{K}_2}\lesssim&\eta(\norm{\p_x^2\ut}^2+\norm{\p_x\ut}^2) +C_\eta(\norm{\p_x\nt}^2+\norm{\nt}_{L^\infty}^2\norm{\p_x^2u^\sha}^2 +\norm{\p_xn^\sha}_{L^\infty}^2\norm{\p_x\ut}^2)\\
	&+C_\eta(\norm{\nt}_{L^\infty}^2\norm{\p_xn^\sha}_{L^\infty}^2\norm{\p_xn^\sha}_{L^2}^2+\norm{\nt}_{L^\infty}^2\norm{\p_xn^\sha}_{L^\infty}^2\norm{\p_x\nt}^2)\\
	\lesssim&\eta(\norm{\p_x^2\ut}^2+\norm{\ut_x}^2)+C_\eta\norm{\nt_x}^2+C_\eta\norm{[\nt,\ut]}_{H^1}^2(\epsilon(1+t)^{-2}+\nu e^{-\alpha t}).
\end{align*}
Note that
\begin{align*}
	\frac{3}{2}(u^\sha_x,(\ut_x)^2)=\frac{3}{2}(u^r_x,(\ut_x)^2)+\frac{3}{2}((u^\sha-u^r)_x,(\ut_x)^2)\gtrsim -\nu e^{-\alpha t}\|\p_x\ut\|^2.
\end{align*}
Then, by integrating \eqref{2nd-rare-1} with respect to time on $[0,t]$ and plugging estimates of $\textbf{K}_i (i=1,\cdots,7)$, we have that for small enough $\eta$,
\begin{equation}
	\label{ineq-2nd-rare'}
	\begin{aligned}
		&\sup\limits_{t\in[0,T]} \norm{\ut_x}^2+
		\int_{0}^{T} \norm{ \p_x^2\ut }^2 \, dt  \\
		&\qquad
		\lesssim  \norm{[\nt_0, \ut_0, \phit_0]}_{H^1}^2
		+ \int_{0}^{T} (\epsilon(1+t)^{-2}+\nu e^{-\alpha t}) \norm{ [\nt,\ut] }_{H^1}^2 \, dt   \\
		&\qquad\quad
		+  \int_{0}^{T} \norm{ [\nt_x,\ut_x,\p_x^2 \phit] }^2 \, dt
		+ \epsilon + \nu .
	\end{aligned}
\end{equation}
Therefore, \eqref{ineq-2nd-rare} follows by using \eqref{ineq-0th-rare} and \eqref{ineq-1st-rare} with small $\eta$, $\nu$ and $\varepsilon$.
\end{proof}

\begin{lemma}
	\label{lem-phi-rare}
	Under the assumptions of Proposition \ref{prop-apriori-rare}, there exist positive constants $\nu_0>0$, $\varepsilon_0>0$, $0<\theta<1$ independent of $T$ such that if $\nu<\nu_0$, and \eqref{eqn-assum-rare} holds, then
	\begin{equation}
		\label{ineq-phi-rare}
		\begin{aligned}
			\norm{\phit(t)}_{H^1}^2\lesssim\norm{\nt(t)}^2+ \epsilon (1+t)^{-2} +\nu e^{-\alpha t},\quad\forall t\in[0,T].
		\end{aligned}
	\end{equation}
\end{lemma}
\begin{proof}
	The Poisson equation \subeqref{eqn-pert-rare}{3} implies that for any $t\geq 0$
	\begin{align*}
		\norm{\phit(t)}_{H^1}^2&\lesssim \norm{\nt(t)}^2+ \norm{[\p_x^2 \phi^r(t),k_3]}^2\\
		& \lesssim \norm{\nt(t)}^2+ \epsilon (1+t)^{-2} +\nu e^{-\alpha t},
	\end{align*}
    where the estimate of $\norm{\p_x^2 \phi^r}$ is presented in Lemma \ref{lem-rare-prop}.
	Therefore, \eqref{ineq-phi-rare} is proved, and thus in particular, 
	\begin{equation}
		\label{ineq-phi-rare0}
		\norm{\phit(0)}_{H^1}^2\lesssim\norm{\nt_0}^2+ \epsilon +\nu.
	\end{equation}
\end{proof}

Now we will complete the proof the Theorem \ref{thm-rare}.
\begin{proof}[Proof of Proposition \ref{prop-apriori-rare} and Theorem \ref{thm-rare}]
	Proposition \ref{prop-apriori-rare} is proved by combining the estimates \eqref{ineq-0th-rare}, \eqref{ineq-1st-rare}, \eqref{ineq-2nd-rare}, and \eqref{ineq-phi-rare0}. Hence, Theorem \ref{thm-shock} follows from the standard local-existence theory, Proposition \ref{prop-apriori-shock}, (iii) of Lemma \ref{lem-rare-prop}, and the fact of the ansatz
	\begin{equation}
		\norm{[n^\sha,m^\sha,\phi^\sha](\cdot,t) - [n^r,m^r,\phi^r](\cdot/t) }_{L^\infty(\R)} \leq \nu e^{-\alpha t}.
	\end{equation}
\end{proof}

%

\section{Appendix}	
\subsection{Proof of Lemma \ref{Lem-periodic}}
Lemma \ref{Lem-periodic} could be obtained by the estimates \eqref{decay-per} for the solution $[n, m, \phi]$ to \eqref{NSP-Ion} and \eqref{initial-periodic}, and the local existence follows from the standard iteration method, which is omitted here for simplicity.

Thanks to the conservation form of \eqref{initial-periodic}, it holds that 
$$ \int_{0}^{\pi} [n-\nb, m-\mb](x,t) dx =0, \quad \text{for~} t\geq0.$$
It remains to prove the exponential decay the oscillations in \eqref{decay-per}.

Now we define the oscillations as 
\begin{equation}
	\label{def-pert-periodic}
	[\nt,\mt,\ut,\phit]=[n-\nb,m-\mb,u-\ub,\phi-\phib].
\end{equation}
The perturbations $[\nt,\ut,\phit]$ satisfy the same equations as \eqref{eqn-pert-rare}, if replacing both $[n^\sha,u^\sha,\phi^\sha]$ and $[n^r,u^r,\phi^r]$ by $[\nb,\ub,\phib]$ and set $k_1=k_2=k_3=0$:
\begin{equation}
	\label{eqn-pert-periodic}
	\begin{cases}
		\p_t\nt+\p_x(nu)=0,\\
		\p_t\ut+u\p_xu+A\left(\dfrac{\p_xn}{n}-\dfrac{\p_x \nb}{\nb}\right)-\p_x\phit
		=\dfrac{\p_x^2 u}{n},\\
		\p_x^2\phit=\nt+e^{-\phib}(1-e^{-\phit}).
	\end{cases}
\end{equation} 

Only in this subsection, without indications on the sets, the spaces are regarded as the ones of the functions on $(0,\pi)$ and $\norm{\cdot}$ denotes for $L^2$-norm on $(0,\pi)$.

\textbf{Step 1.} Similarly as the basic energy estimates for the perturbation on the rarefaction wave (Lemma \ref{lem-0th-rare}), taking the $L^2$-inner products between  \subeqref{eqn-pert-periodic}{2} and $n\ut$ as well as \subeqref{NSP-Ion}{1} and $\Lambda(n,\nb):=A\int^{n}_{\nb}\frac{s-\nb}{s^2}ds$ on $(0,\pi)$, and summing the resulting equation together, we have
\begin{align}\label{0th-periodic-energy}
	&\frac{d}{dt}\left\{\frac{1}{2}\left(n\ut,\ut\right)+\left(n,\Lambda\right)\right\}-\left(\phit_x,n\ut\right)+(\ut_x,\ut_x)=0,
\end{align} 
where we have used the following identity
\begin{align*}
	(n,-\Lambda_t)+((nu)_x,\Lambda)+A\left(\frac{n_x}{n},n\ut\right)=0.
\end{align*}
For the third term on the left hand side of \eqref{0th-periodic-energy}, we utilize integration by parts, \subeqref{eqn-pert-periodic}{1} and \subeqref{eqn-pert-periodic}{3} to obtain that
\begin{align*}
	&-\left(\phit_x,n\ut\right)=\frac{1}{2}\frac{d}{dt}\left(\phit_x,\phit_x\right)+\left(\phit,\nb\p_t\left(1-e^{-\phit}\right)\right)+\left(\phit,-\left(\nt \ub\right)_x\right),
\end{align*} 
and 
\begin{align*}
	&\quad \left(\phit,\nb\p_t\left(1-e^{-\phit}\right)\right)-\left(\phit,\left(\nt \ub\right)_x\right)
	=\left(\phit,\nb\p_t\left(1-e^{-\phit}\right)\right)+\left(\phit_x,\nt u^\sha\right)\\
	&=\left(\phit,\nb\p_t\left(1-e^{-\phit}\right)\right)-\left(\phit_x\left(1-e^{-\phit}\right),\nb \ub\right)+\ub\left(\phit_x,\p_x^2\phit \right)\\
	&=\frac{d}{dt}\left(\nb,1-e^{-\phit}-\phit e^{-\phit}\right).
\end{align*}  
Therefore, we obtain that
\begin{equation}\label{0th-periodic-1}
	\frac{d}{dt}\left\{
	\frac{1}{2}\left(n\ut,\ut\right)+\left(n,\Lambda\right)
	+\frac{1}{2}(\phit_x,\phit_x)
	+(\nb,1-e^{-\phit}-\phit e^{-\phit})
	\right\}+(\ut_x,\ut_x)=0.
\end{equation} 
Note that $\Lambda(n,\nb)\sim|\nt|^2$ if $|\nt|$ is bounded, and 
\begin{align*}
	(\nb,1-e^{-\phit}-\phit e^{-\phit})
	&\gtrsim\nb \|\phit\|^2,
\end{align*}
for any $\|\phit\|_{L^\infty}\leq C$. Then, integrating \eqref{0th-periodic-1} in time over $[0,t]$ gives that for small $\nu$
\begin{equation}
	\label{ineq-0th-periodic'}
	\sup\limits_{t\in[0,T]} \norm{[\ut, \nt,\phit,\phit_x]}^2+
	\int_{0}^{T} \norm{ \ut_x }^2 \, dt  
	\lesssim  \norm{[\nt_0,\ut_0]}^2 + \norm{ \phit_0}_{H^1}^2.
\end{equation}

\textbf{Step 2.} Similarly as the energy estimates for the perturbation on the rarefaction wave (Lemma \ref{lem-1st-rare}),
differentiating \subeqref{eqn-pert-periodic}{1} and \subeqref{eqn-pert-periodic}{3} in $x$, taking the inner products of the resultant with $\nt_x/n^2$ and $\phit_x$ with respect to $x$ respectively on $(0,\pi)$, and using integration by parts yield that
\begin{align}
	\label{1st-periodic-1}
	\frac{1}{2}\frac{d}{dt}(\nt_x^2,n^{-2}) +(\p_x^2\phit,\p_x^2\phit) +(\nb e^{-\phit}\phit_x,\phit_x)+\left(\p_x^2\ut,\frac{\nt_x}{n}\right)+(\nt_x,\phit_x)=-\frac{1}{2}\left(\frac{\ut_x}{n^2},\nt_x^2\right),
\end{align} 
where we have used 
\begin{align*}
	-(\nt_x^2,n^{-3}\p_tn)-(u_x\nt_x+u\p_x^2\nt+n_x\ut_x,\frac{\nt_x}{n^2})=-\frac{1}{2}\left(\frac{\ut_x}{n^2},\nt_x^2\right).
\end{align*}
By taking the inner product of \subeqref{eqn-pert-periodic}{2} with $\nt_x$ on $(0,\pi)$ and sum with \eqref{1st-periodic-1}, we obtain that
\begin{align}
	\label{1st-periodic-2}
	&\frac{d}{dt}(\ut,\nt_x)+\frac{1}{2}\frac{d}{dt}(\nt_x^2,n^{-2})+(\p_x^2\phit,\p_x^2\phit)+(\nb e^{-\phit}\phit_x,\phit_x)+A(n^{-1},\nt_x^2)\nonumber\\
	&=-\underbrace{\frac{1}{2}\left(\frac{\ut_x}{n^2},\nt_x^2\right)}_{\textbf{J}_1}+\underbrace{(\ut,\p_t\nt_x)-(uu_x)}_{\textbf{J}_2}.
\end{align}
It follows from the \textit{a priori} assumption 
\begin{equation}
	\label{eqn-assum-periodic}
	\varepsilon:=\sup\limits_{t\in[0,T]} \norm{[\nt, \ut, \phit]}_{H^1} < 2\nu,
\end{equation} 
the Young inequality and Sobolev embedding that
\begin{align*}
	\abs{\textbf{J}_1}&\lesssim\eta\norm{\nt_x}^2+C_\eta\norm{\nt_x}^2\norm{\ut_x}_{L^\infty}^2\lesssim(\eta+C_\eta\varepsilon^2)\norm{\nt_x}^2+C_\eta\varepsilon^2\norm{\p_x^2\ut}^2.
\end{align*}
To estimate $\textbf{J}_2$, we utilize integration by parts and \subeqref{eqn-pert-periodic}{1} to derive that
\begin{align*}
	\abs{\textbf{J}_2}&=\abs{(\ut_x,\p_x(nu))-(uu_x,\nt_x)}=\abs{(n,\ut_x^2)}\lesssim\norm{\ut_x}^2.
\end{align*}
Note that
\begin{align*}
	\left|\int_{0}^{T}\frac{d}{dt}(\ut,\nt_x)\right|\lesssim \|\ut_0\|^2+\|\nt_{0x}\|^2+\eta\|\nt_x(T)\|^2+C_\eta\|\ut(T)\|^2, \quad (\nb e^{-\phit}\phit_x,\phit_x)\gtrsim \|\phit_x\|^2.
\end{align*}
Then, by integrating \eqref{1st-periodic-2} with respect to time on $[0,t]$ and plugging estimates of $\textbf{J}_1$ and $\textbf{J}_2$, we have that for small $\eta$ and $\nu$,
\begin{equation}
	\label{ineq-1st-periodic'}
	\begin{aligned}
		\sup\limits_{t\in[0,T]} \norm{\nt_x}^2+
		\int_{0}^{T} \norm{ [\nt_x, \phit_x, \p_x^2\phit] }^2 \, dt  \lesssim & \norm{[\ut_0]}^2 + \norm{ [\nt_0,\phit_0]}_{H^1}^2  +\sup\limits_{t\in[0,T]} \norm{\ut}^2   \\
		&
		+  \int_{0}^{T} \norm{ \p_x\ut }^2 \, dt
		+  \varepsilon\int_{0}^{T} \norm{ [\p_x \nt, \p_x^2\ut]}^2 \, dt .
	\end{aligned}
\end{equation}

\textbf{Step 3.} Similarly as the energy estimates for the perturbation on the rarefaction wave (Lemma \ref{lem-2nd-rare}),
differentiating \subeqref{eqn-pert-periodic}{2} in $x$, taking inner products of the resultant with $\p_x\ut$ with respect to $x$ on $(0,\pi)$, and using integration by parts yield that
\begin{equation}\label{2nd-periodic-1}
	\frac{1}{2}\frac{d}{d}(\ut_x,\ut_x)+\left(\frac{\p_x^2\ut}{n},\p_x^2\ut\right)=\underbrace{(\ut\p_x^2\ut,\ut_x)}_{\textbf{K}_1}+\underbrace{A\left(\frac{n_x}{n},\p_x^2\ut\right)}_{\textbf{K}_2}+\underbrace{\left(\p_x^2\phit,\ut_x\right)}_{\textbf{K}_3}.
\end{equation}
It is obvious to get
\begin{align*}
	\abs{\textbf{K}_1}&\lesssim\eta\norm{\p_x^2\ut}^2+C_\eta\norm{\ut}_{H^1}^2\norm{\ut_x}^2,\\
	\abs{\textbf{K}_2}&\lesssim\eta(\norm{\p_x^2\ut}^2)+C_\eta\norm{\nt_x}^2,\\
	\abs{\textbf{K}_3}&\lesssim\norm{\ut_x}^2+\norm{\p_x^2\phit}.
\end{align*}
Then, by integrating \eqref{2nd-rare-1} with respect to time on $[0,t]$ and plugging estimates of $\textbf{K}_i (i=1,\cdots,3)$, we have that for small $\eta$ and $\nu$,
\begin{equation}
	\label{ineq-2nd-periodic'}
	\begin{aligned}
		&\sup\limits_{t\in[0,T]} \norm{\ut_x}^2+
		\int_{0}^{T} \norm{ \p_x^2\ut }^2 \, dt  
		\lesssim  \norm{[\nt_0, \ut_0, \phit_0]}_{H^1}^2
		+  \int_{0}^{T} \norm{ [\nt_x,\ut_x,\p_x^2 \phit] }^2 \, dt.
	\end{aligned}
\end{equation}

\textbf{Step 4.} Combining the estimates \eqref{ineq-0th-periodic'}, \eqref{ineq-1st-periodic'}, \eqref{ineq-2nd-periodic'}, and the elliptic estimate from \subeqref{eqn-pert-rare}{3}:
\begin{align*}
	\norm{\phit(t)}_{H^1}^2&\lesssim \norm{\nt(t)}^2, \quad \text{for~}t\geq 0,
\end{align*}
we have that there exist $\nu_0>0$ independent of $T$ such that if $\nu<\nu_0$ and \eqref{eqn-assum-periodic} holds, then
\begin{equation}
	\label{eqn-aprioi-periodic}
	\sup\limits_{t\in[0,T]} \norm{[\nt, \ut, \phit]}_{H^1}^2+
	\int_{0}^{T} \Big(\norm{[\nt_x, \ut_x, \p_x^2\ut, \phit_x,\p_x^2\phit]}^2 \Big) \, dt
	\lesssim  \norm{[\nt_0,\ut_0]}_{H^1}^2.
\end{equation}
It follows from \eqref{eqn-aprioi-periodic} and Poincar\'{e} inequality that there exist constants $ C>0 $ and $ \alpha >0 $ independent of $ \nu $ and $ t$ such that for $t\geq0$,
\begin{equation}
	\norm{[n,m,u,\phi]-[\nb,\mb,\ub,\phib]}_{H^1((0,\pi))}(t) \leq C \nu e^{-2\alpha t}. 
\end{equation}
The estimates for the higher order derivatives in \eqref{eqn-pert-periodic} are standard and their exponential decay rates can be proved in the same way with the aid of Poincar\'e inequality, and the details are omitted here.

\subsection{Proof of Lemma \ref{lem-error}}
\begin{enumerate}[label = \rm{(\roman*)},ref =\rm{(\roman*)}]
\item 
	Define $$a(u,v)=\frac{f(u)-f(v)}{u-v}=\int_{0}^{1} f'(v+\theta(u-v)) d\theta.$$
	$$b(u,v)=\frac{f'(u)-f'(v)}{u-v}=\int_{0}^{1} f''(v+\theta(u-v)) d\theta.$$
	We first consider \eqref{lem-error-1}:
	\begin{align*}
	&f(u^\sha)-f(u^*)-(f(u^-)-f(\ub^-))(1-g_1)-(f(u^+)-f(\ub^+))g_1\\
	=\,& \{ f(u^\sha)-f(u^*)-f(u^-)+f(\ub^-) \}(1-g_1) \\
	&\quad + \{  f(u^\sha)-f(u^*)-f(u^+)+f(\ub^+) \} g_1.
	\end{align*}
	It suffices to prove that one part of the above formula equals to $\sR$, and the proof of the other part is similar.
	Noting that $v^\pm$ belong to $\sR_{\rm Ti}$, and $g, g_0,$ and $ g_1$ all belong to $\sG$ (i.e. $g(1-g)$ etc. belong to $\sR_{\rm Sp}$), it yields that
	\begin{align*}
	& \{  f(u^\sha)-f(u^*)-f(u^-)+f(\ub^-) \}(1-g_1) \\
	=\,& \{a(u^\sha, u^*) (v^-(1-g)+v^+ g) - a(u^-, \ub^-) v^- + \sR \} (1-g_1)\\
	=\,& \{a(u^\sha, u^*) - a(u^-, \ub^-)  \} v^- (1-g)(1-g_1)+\sR\\
	=\,& \int_{0}^{1}  \{f'(u^*+\theta(u^\sha- u^*) ) -f'(\ub^-+\theta(u^\sha- u^*) )\}  + \{f'(\ub^-+\theta(u^\sha- u^*) ) \\
	& \qquad- f'(\ub^-+\theta(u^- - \ub^-) )\}   d\theta \cdot v^- (1-g)(1-g_1)+\sR\\
	=\,& \int_{0}^{1}  \big\{b\big(u^*+\theta(u^\sha- u^*), \ub^-+\theta(u^\sha- u^*) \big) (u^*-\ub^-)   \\
	& \quad  + b\big(\ub^-+\theta(u^\sha- u^*) , \ub^-+\theta(u^- - \ub^-) \big) \theta(u^\sha- u^*-u^- + \ub^-)\big\}d\theta  \\
	&\qquad \cdot v^- (1-g)(1-g_1)+\sR \\
	=\,&\sR,
	\end{align*}
	where $u^*-\ub^-=(\ub^+ - \ub^-) g_0$, $u^\sha- u^*-u^- + \ub^-=(v^+ - v^-) g$ are applied in the last step. Therefore, \eqref{lem-error-1} is proved.
	
	Moreover, \begin{align*}
		& f(u^*) - f(\ub^-)(1-g_1) - f(\ub^+)g_1\\
		=\,& \{f(u^*) - f(\ub^-)\} (1-g_1) + \{f(u^*) - f(\ub^+)\} g_1\\
		=\,& a(u^*,\ub^-) (u^* - \ub^-) (1-g_1) + a(u^*,\ub^+) (u^* - \ub^+) g_1\\
		=\,& a(u^*,\ub^-) (\ub^+ - \ub^-) g_0(1-g_1) - a(u^*,\ub^+) (\ub^+ - \ub^-) (1-g_0) g_1 \\
		=\,&\sR_{\rm Sp}.
	\end{align*}
	The above term are exactly $0$ when $g_1=\dfrac{f(u^*)-f(\ub^-)}{f(\ub^+)-f(\ub^-)}$. Therefore, \eqref{lem-error-2} is proved by summing up \eqref{lem-error-1} and the above formula.

\item \eqref{lem-error-3} is obtained directly as follows.
\begin{align*}
	&\p_xu^\sha-\p_xu^*-\p_x u^-(1-g_1)-\p_x u^+ g_1\\
	=\,&\p_x (v^+ - v^-) (g-g_1) + (v^+ - v^-) \p_x g + \p_x \sR \\
	=\,&\sR.
\end{align*}

\item \eqref{lem-error-3.5} is trivial. \eqref{lem-error-4} is obtained directly as follows.
\begin{align*}
	&n^\sha u^\sha-n^*u^*-(n^- u^- -\nb^-\ub^-)(1-g_1)-(n^+ u^+ -\nb^+\ub^+)g_1\\
 =\,& (n^\sha u^\sha-n^*u^*-n^- u^- +\nb^-\ub^-)(1-g_1) + \cdots \\
 =\,& (n^\sha (u^\sha -u^*)+ (n^\sha-n^*)u^* - n^-v^- - \rho^- \ub^-)(1-g_1)+\sR+ \cdots\\
 =\,&  (n^\sha v^-(1-g)+ \rho^-(1-\tilde{g})u^* - n^-v^- - \rho^- \ub^-)(1-g_1)+\sR+ \cdots\\
 =\,&  \Big( \big(\nb^-(1-\tilde{g}_0)+\rho^-(1-\tilde{g})\big) v^-(1-g)- n^-v^-\\
 &\quad + \rho^-(1-\tilde{g})\ub^-(1-g_0)  - \rho^- \ub^-\Big)(1-g_1)+\sR+ \cdots\\
 =\,& \sR,
\end{align*}
where in the last two steps, $n^\sha$, $u^\sha$, $n^*$, and $u^*$ are replacing by the constants, the decaying terms in $\sR_{\rm Sp}$ and the weights with the aid of their definitions. 
%
\end{enumerate}

\subsection{Derivation of \eqref{eqn-pert-shock}-\eqref{eqn-error-shock}}\label{subsect-error}
It follows from \eqref{weight-shock}, \eqref{eqn-shock-deriv} and the property that $\abs{\X}, \abs{\Y}\leq C$ for all $t\geq 0$, we have $\sigma_{\X},\sigma_{\Y}\in\sG$. Note  that $m^\sha$ in \eqref{def-ansatz-shock} satisfies 
	\begin{equation}
		m^\sha=m^s_\Y+w^-(1-\sigma_{\Y})+w^+\sigma_{\Y}=m^s_\Y+w^-(1-\sigma_{\X})+w^+\sigma_{\X}+\sR.
	\end{equation}
	Then, one has
	\begin{equation}
		\p_x m^\sha=\p_x m^s_\X+\p_x(m^s_\Y-m^s_\X)+\p_x w^-(1-\sigma_{\X})+\p_xw^+\sigma_\X+(w^+-w^-)\p_x\sigma_{\X}+\p_x\sR.
	\end{equation}
On the other hand, for the ansatz $n^\sha$,
\begin{equation}
	\p_tn^\sha=\p_xn^s_\X(-s-\X')+\p_t\rho^-(1-\sigma_{\X})+\p_t\rho^+\sigma_{\X}+(\rho^+-\rho^-)\p_x\sigma_{\X}(-s-\X').
\end{equation}
By using the fact
\begin{equation*}
	\p_xm^s_\X=(\mb^+-\mb^-)\p_x\sigma_{\X},\quad \p_xn^s_\X=(\nb^+-\nb^-)\p_x\sigma_{\X},
\end{equation*}
and the equations of periodic solutions $[n^\pm,m^\pm,\phi^\pm]$, we have \subeqref{eqn-pert-shock}{1} with $h_1$ in \subeqref{eqn-error-shock}{1}.

Similarly, note that $n^\sha$ and $\phi^\sha$ in \eqref{def-ansatz-shock} satisfy
\begin{align}
	n^\sha&=n^s_{\X}+\rho^-(1-\sigma_{\X})+\rho^+\sigma_{\X}=n^s_{\X}+\rho^-(1-\sigma_{\Y})+\rho^+\sigma_{\Y}+\sR,\\
	\phi^\sha&=\phi^s_{\X}+\varphi^-(1-\sigma_{\X})+\varphi^+\sigma_{\X}=\phi^s_{\X}+\varphi^-(1-\sigma_{\Y})+\varphi^+\sigma_{\Y}+\sR.
\end{align}
By using Lemma \ref{lem-error},  we have
\begin{align}\label{6-19}
	\frac{(m^\sha)^2}{n^\sha}&=\frac{(m^s_\Y)^2}{n^s_\X}+\left(\frac{(m^-)^2}{n^-}-\frac{(\mb^-)^2}{\nb^-}\right)(1-\sigma_{\Y})+\left(\frac{(m^+)^2}{n^+}-\frac{(\mb^+)^2}{\nb^+}\right)\sigma_{\Y}+\sR \nonumber\\
	&=\frac{(m^s_\Y)^2}{n^s_\Y}+\left(\frac{(m^s_\Y)^2}{n^s_\X}-\frac{(m^s_\Y)^2}{n^s_\Y}\right)+\left(\frac{(m^-)^2}{n^-}-\frac{(\mb^-)^2}{\nb^-}\right) (1-\sigma_{\Y})\\
	&\qquad+\left(\frac{(m^+)^2}{n^+}-\frac{(\mb^+)^2}{\nb^+}\right)\sigma_{\Y}+\sR, \nonumber
\end{align}
\begin{align}
	(A+1)n^\sha=(A+1)\{n^s_\Y+(n^s_\X-n^s_\Y)+\rho^-(1-\sigma_{\Y})+\rho^+\sigma_{\Y}\}+\sR,
\end{align}
\begin{align}\label{6-20}
	\p_x\frac{m^\sha}{n^\sha}&=\p_x\frac{m^s_\Y}{n^s_\X}+\p_x\frac{m^-}{n^-}(1-\sigma_{\Y})+\p_x\frac{m^+}{n^+}\sigma_{\Y}+\sR\nonumber\\
	&=\p_x\frac{m^s_\Y}{n^s_\Y}+\p_x(\frac{m^s_\Y}{n^s_\X}-\frac{m^s_\Y}{n^s_\Y})+\p_x\frac{m^-}{n^-}(1-\sigma_{\Y})+\p_x\frac{m^+}{n^+}\sigma_{\Y}+\sR,
\end{align}
\begin{align}\label{6-21}
	\frac{1}{2}(\p_x\phi^\sha)^2&=\frac{1}{2}\{(\p_x\phi^s_{\X})^2 +(\p_x\phi^-)^2(1-\sigma_{\Y}) +(\p_x\phi^+)^2\sigma_{\Y}\}+\sR\nonumber\\
	&=\frac{1}{2}\{(\p_x\phi^s_{\Y})^2+(\p_x\phi^s_{\X})^2-(\p_x\phi^s_{\Y})^2 +(\p_x\phi^-)^2(1-\sigma_{\Y})  +(\p_x\phi^+)^2\sigma_{\Y}\}+\sR,
\end{align}
and
\begin{align}\label{6-22}
	\p_x^2\phi^\sha&=\p_x^2\phi^s_{\X}+\p_x^2\phi^-(1-\sigma_{\Y})+\p_x^2\phi^+\sigma_{\Y}+\sR\nonumber\\
	&=\p_x^2\phi^s_{\Y}+\p_x^2\phi^s_{\X}-\p_x^2\phi^s_{\Y}+\p_x^2\phi^-(1-\sigma_{\Y})+\p_x^2\phi^+\sigma_{\Y}+\sR.
\end{align}
Moreover,
\begin{align}\label{6-23}
	\p_t m^\sha=\p_xm^s_\Y(-s-\Y')+\p_tm^-(1-\sigma_{\Y})+\p_tm^+\sigma_{\Y}+(w^+-w^-)\p_x\sigma_{\Y}(-s-\Y').
\end{align}
Taking derivatives of \eqref{6-19}-\eqref{6-22} with respect to $x$ and adding with \eqref{6-23} yield
\begin{align*}
	\p_tm^\sha+\p_x\left(\frac{(m^\sha)^2}{n^\sha}+(A+1)n^\sha\right)-\p_x^2\left(\frac{m^\sha}{n^\sha}\right)-\p_x\left(\frac{1}{2}(\p_x\phi^\sha)^2-\p_x^2\phi^\sha\right)=h_2,
\end{align*}
where we have used equations of viscous shock waves $[n^s,u^s,\phi^s]$, periodic solutions $[n^\pm,$ $m^\pm, \phi^\pm]$ and the Rankine-Hugoniot condition \subeqref{RH}{2}:
\begin{align*}
-\p_xm^s_\Y \Y' +\jump{w}\p_x\sigma_{\Y}(-s-\Y')
=&-\jump{m} \p_x\sigma_{\Y}(s+\Y') + \jump{\mb}\p_x \sigma_{\Y}s\\
=&-\jump{m} \p_x\sigma_{\Y}(s+\Y') + \jump{\frac{\mb^2}{\nb}+(A+1)\nb} \p_x\sigma_{\Y}.
\end{align*}
Finally, the last equation and $h_3$ in \eqref{eqn-pert-shock} directly follows from Lemma \ref{lem-error}.

\

\noindent {\bf Acknowledgments:}
The research of Yeping Li is supported in part by the National Natural Science Foundation of China (Grant No. 12171258 and 12331007). The research of Yu Mei is supported by the National Natural Science Foundation of China No. 12101496 and the Fundamental Research Funds for the Central Universities No. G2021KY05101.	
The research of Yuan Yuan is supported by the Natural Science Foundation of Guangdong Province (No. 2021A1515010247), and National Key R\&D Program of China (No. 2021YFA1002900).

\bibliographystyle{plain}
\bibliography{nsp-ref.bib}
	
%
%
%
%
%
%
%

\end{document}